\documentclass[a4paper,reqno, 12pt]{amsart}
\usepackage{graphicx,color}
\usepackage{footmisc}
\usepackage{amsmath}
\usepackage{amsthm}
\usepackage{amsfonts}
\usepackage{amssymb}
\usepackage{comment}
\usepackage{tikz}
\usepackage{stix} 

\definecolor{darkblue}{rgb}{0.1,0.1,0.7}

\def\cyril#1{\color{darkgreen}{ [CL: #1] }\color{black}}
\definecolor{darkgreen}{rgb}{0.1,0.7,0.1}

\def\fabio#1{\color{darkred}{ [FT: #1] }\color{black}}
\definecolor{darkred}{rgb}{0.7,0.1,0.1}

\usepackage[utf8]{inputenc}
\numberwithin{equation}{section}
\usepackage{hyperref}
\hypersetup{
    colorlinks=true,
    linkcolor=black,
    citecolor=blue,
    urlcolor=blue,
    pdfborder={0 0 0}
}
\usepackage{cleveref}

\def\Lip{\textsl{Lip}}

\newcommand{\ttv}{\mathtt v}
\newcommand{\ttw}{\mathtt w}
\newcommand{\ttu}{\mathtt u}
\newcommand{\tu}{\mathtt u}
\newcommand{\tv}{\mathtt v}
\newcommand{\Tw}{\mathtt w}

\newcommand{\bbC}{\mathbb{C}}
\newcommand{\bbD}{\mathbb{D}}
\newcommand{\bbE}{\mathbb{E}}

\newcommand{\bbN}{\mathbb{N}}

\newcommand{\bbP}{\mathbb{P}}

\newcommand{\bbR}{\mathbb{R}}

\newcommand{\bbZ}{\mathbb{Z}}

\newcommand{\cC}{\mathcal{C}}
\newcommand{\cD}{\mathcal{D}}
\newcommand{\cE}{\mathcal{E}}
\newcommand{\cF}{\mathcal{F}}

\newcommand{\cJ}{\mathcal{J}}

\newcommand{\cL}{\mathcal{L}}

\newcommand{\cN}{\mathcal{N}}

\newcommand{\cR}{\mathcal{R}}
\newcommand{\cS}{\mathcal{S}}

\newcommand{\R}{\mathbb{R}}

\newcommand{\C}{\mathbb{C}}
\newcommand{\N}{\mathbb{N}}
\newcommand{\E}{\mathbb{E}}
\renewcommand{\P}{\mathbb{P}}

\DeclareMathOperator{\Var}{Var}
\DeclareMathOperator{\Cov}{Cov}
\DeclareMathOperator{\EE}{E}

\newtheorem{theorem}{Theorem}[section]

\newtheorem{remark}[theorem]{Remark}

\newtheorem{assumption}[theorem]{Assumption}

\newtheorem{lemma}[theorem]{Lemma}

\newtheorem{proposition}[theorem]{Proposition}

  \crefname{theorem}{Theorem}{Theorems}
  \crefname{lemma}{Lemma}{Lemmas}
  \crefname{remark}{Remark}{Remarks}
  \crefname{proposition}{Proposition}{Propositions}
\crefname{notation}{Notation}{Notations}
\crefname{claim}{Claim}{Claims}
  \crefname{definition}{Definition}{Definitions}
  \crefname{corollary}{Corollary}{Corollaries}
  \crefname{section}{Section}{Sections}
  \crefname{figure}{Figure}{Figures}
    \crefname{assumption}{Assumption}{Assumptions}

\def\tun{\mathbf{1}} 

\newcommand{\mBB}{m_{\text{\tiny{bridge}}}}
\newcommand{\mBE}{m_{\text{\tiny{exc}}}}

\renewcommand{\[}{\begin{equation}}
\renewcommand{\]}{\end{equation}}

\title{Convergence of dynamical stationary fluctuations}

\author{Cyril Labb\'e}
\address{Universit\'e Paris Cit\'e, Laboratoire de Probabilit\'es, Statistique et Mod\'elisation, UMR 8001, F-75205 Paris, France and Institut Universitaire de France (IUF).}
\email{clabbe@lpsm.paris}

\author{Beno\^it Laslier}
\address{Universit\'e Paris Cit\'e, Laboratoire de Probabilit\'es, Statistique et Mod\'elisation, UMR 8001, F-75205 Paris, France}
\email{laslier@lpsm.paris}

\author{Fabio Toninelli}
\address{TU Wien, Faculty of Mathematics and Geoinformation, Wiedner Hauptstra\ss e 8-10, A-1040 Vienna, Austria}
\email{fabio.toninelli@tuwien.ac.at}

\author{Lorenzo Zambotti}
\address{Sorbonne Universit\'e, Laboratoire de Probabilit\'es, Statistique et Mod\'elisation, UMR 8001, F-75205 Paris, France}
\email{zambotti@lpsm.paris}

\date{}

\begin{document}
	
\maketitle

\begin{abstract}
  We present a general \emph{black box} theorem that ensures
  convergence of a sequence of stationary Markov processes, provided a
  few assumptions are satisfied. This theorem relies on a control of
  the resolvents of the sequence of Markov processes, and on a
  suitable characterization of the resolvents of the limit. One major
  advantage of this approach is that it circumvents the use of the
  Boltzmann-Gibbs principle: for instance, we deduce in a rather
  simple way that the stationary fluctuations of the one-dimensional
  zero-range process converge to the stochastic heat equation. More importantly, it
  allows to establish results that were probably out of reach of
  existing methods: using the black box result, we are able to prove
  that the stationary fluctuations of a discrete model of ordered
  interfaces, that was considered previously in the statistical
  physics literature, converge to a system of reflected stochastic
  PDEs.
	
\end{abstract}

\setcounter{tocdepth}{1}
\tableofcontents

\section{Introduction}

This work deals with scaling limits of stationary fluctuations of Markov processes. We typically think of statistical physics models such as particle systems (exclusion process, zero-range process), evolving interfaces ($\nabla \phi$ interfaces, height functions) or dimer models, and the limiting process is usually given by a stochastic PDE. There exists by now a large literature on this topic. However, establishing rigorously such a convergence result often requires a delicate estimate, the so-called Boltzmann-Gibbs principle, that allows to replace space-time averages of discrete observables by a function of the local density up to some negligible error. We refer to~\cite[Sec 11.1]{KipLan} for more details on the Boltzmann-Gibbs principle in a reversible setting, ~\cite{ChangLandimOlla} in a non-reversible setting and also~\cite{GoncalvesJara} for an extension called a second order Boltzmann-Gibbs principle; let us mention also
that the Boltzmann-Gibbs principle, together with the relative entropy method, plays an important role for the
proof of convergence of \emph{non-equilibrium fluctuations} in some interacting particle systems \cite{JM}.
Moreover, there are some interesting models from statistical physics for which the scaling limit of the stationary fluctuations is still out of reach of existing methods.\\

In this article, we present a general theorem, dubbed ``black box theorem'' in the sequel, that ensures convergence of a sequence of stationary Markov processes towards some limiting stationary Markov process, provided a set of five assumptions are fulfilled. Having in mind the models of statistical physics alluded to above, this theorem is designed to circumvent the use of the Boltzmann-Gibbs principle. As we will see, this theorem not only provides simpler proof for existing convergence results (see the example of the zero-range process below) but it also allows to prove new convergence results (see the pair of ordered interfaces below).\\
Our black box theorem relies on a change of perspective on how one characterizes the limiting process. In classical methods, the characterization of the limit goes through a martingale problem and the Boltzmann-Gibbs principle is used for computing the quadratic variation of the scaling limits of the discrete martingales. In our approach, we instead require a characterization of the limit through an identity satisfied by its resolvents, see Assumption \ref{Ass:Charac} below.\\
The remaining four assumptions concern the sequence of stationary Markov processes. We naturally require the sequence of invariant measures to converge to the invariant measure of the limit, Assumption \ref{Ass:State}, and we ask the sequence of Markov processes to be tight, Assumption \ref{Ass:Tight}. The last two assumptions deal specifically with the resolvents of the sequence of Markov processes. In Assumption \ref{Ass:Equi}, we require the sequence of resolvents to be equicontinuous in order to extract converging subsequences. Finally, Assumption \ref{Ass:IbPF} imposes that the sequence of resolvents are characterized by an identity that is asymptotically the same as the identity satisfied by the resolvents of the limit.\\

Our black box theorem finds its origin in an article of the fourth author~\cite{Zam04} on the convergence of the stationary fluctuations of a $\nabla\phi$  interface model with reflection towards the solution of a reflected SPDE introduced by Nualart and Pardoux~\cite{NualartPardoux92}. This convergence had been established by Funaki and Olla \cite{FunaOlla} using the Boltzmann-Gibbs principle: the proof in~\cite{Zam04} allowed to circumvent this delicate step. Later this technique was turned into a "black box" result in \cite{MR2267701},
however this approach was restricted to continuous-in-time approximations, while a paper of the first author \cite{EthLab} proved that the convergence result of \cite{Zam04} was true also for discrete approximations. The present work generalises both techniques to a much
more general setting that encompasses both discrete and continuous state-spaces.\\

In the present paper, we apply our general result to three models of statistical physics. We first consider in Section \ref{sec:0} the one-dimensional zero-range process on a segment and show how our approach circumvents the use of the delicate Boltzmann-Gibbs principle, and simply requires some standard estimates under the invariant measure of the discrete dynamics. We emphasize that the result for the zero-range process is not new (and that it holds under weaker conditions on the transition rates than those we assume below) and that we present it here for pedagogical reasons, that is, to first illustrate our method in a simple context.\\

The second and main application of our general result concerns an
evolving pair of ordered lattice paths $\ttv \ge \ttw$ on $[0,2N]$,
that are bound to the height $0$ at $0$ and $2N$ and evolve according
to a corner-flip dynamics: the corners of $\ttv$ and $\ttw$ flip at
some prescribed rates, except if the flip breaks the ordering of the
two interfaces. This last constraint gives rise to a pair of
\emph{reflection measures} in the evolution equations: one that pushes
up $\ttv$ and one that pushes down $\ttw$. Note that this process is a
particular instance of a model introduced previously in
\cite{luby2001markov,wilson2004mixing} in the context of dynamics of
planar structures and tilings. At a higher level, our
  motivation to look at this process is the aim to understand the scaling limit of the
  evolution of \emph{two-dimensional} random surfaces. The Markov chain
  on a pair of ordered lattice paths is then a toy model where the
  paths represent the level lines of the height function, and only two
  level lines are present. See the end of this section, as well as
  Remark \ref{rem:k>2}, for a natural generalization with an arbitrary
  number $k$ of interacting level lines.

\begin{figure}\centering
	\begin{tikzpicture}[xscale=0.7, yscale=0.7, >=stealth]\label{Fig}
		\draw[-,thin,color=black] (0,0) -> (14,0);

		\draw[-,thick,color=black] (0,0) node[below left] {$0$} -- (1,1) -- (2,0.08) -- (3,-0.92) -- (4,0.08) -- (5,1) -- (6,2) -- (7,1) -- (8,0.08) -- (9,1.08) -- (10,0.08) -- (11,1) -- (12,2) -- (13,1) -- (14,0) node[below right] {$2N$};
		\draw[-,thick,color=black] (0,0)  -- (1,-1) -- (2,0) -- (3,-1) -- (4,0) -- (5,-1) -- (6,-2) -- (7,-1) -- (8,0) -- (9,1) -- (10,0) -- (11,-1) -- (12,-2) -- (13,-1) -- (14,0) node[below right] {$2N$};
		
		%
		%
		%
		%

	\end{tikzpicture}
	\caption{Example of a pair of reflected interfaces: $k=3$ and $k=9$ are \emph{contact points}. At contact points, a reflection term prevents the jumps that would break the ordering to occur.}
\end{figure}
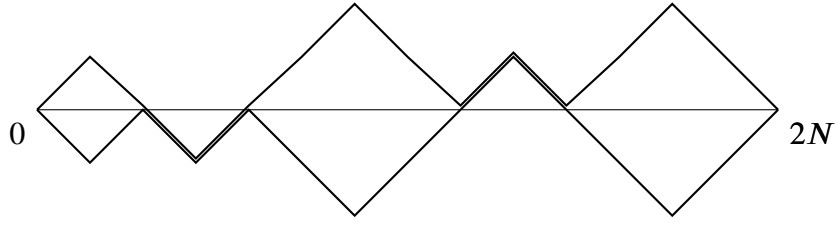

In Section \ref{Sec:Reflected}, we will show that the stationary fluctuations of this model of reflected interfaces, under a diffusive rescaling, converge to the stationary solution of the following system of stochastic PDEs driven by independent space-time white noises $\xi^v$, $\xi^w$
\begin{align}
	\begin{cases}
		\partial_t v = \frac12\partial^2_x v + \xi^v + \frac1{\sqrt 2}\eta\;,\quad x\in [0,1]\;,\quad t\ge 0\;,\\
		\partial_t w = \frac12\partial^2_x w + \xi^w - \frac1{\sqrt 2}\eta\;,\quad x\in [0,1]\;,\quad t\ge 0\;,\\
		v\ge w\;,\quad \eta(dt,dx)\ge 0\;,\quad \int_{\R_+ \times (0,1)} (v-w)(t,x) \eta(dt,dx) = 0\;,
	\end{cases}
\end{align}
In these equations, $\eta$ is a random measure on $\R_+\times (0,1)$, carried by the space-time points where $v$ and $w$ meet, and whose role is to preserve the ordering $v\geq w$. This system of stochastic PDEs is intimately related to the reflected stochastic heat equation introduced by Nualart and Pardoux~\cite{NualartPardoux92}, see Subsection \ref{Sec:SPDEPair}. The main difficulty in the proof of the convergence of the discrete dynamics consists in showing that the two discrete reflection measures converge to the \emph{same} reflection measure $\eta$ in the limit: this question, for a very similar model, was left open in~\cite{EthLab}. Implementing the Boltzmann-Gibbs principle to prove such a statement seems quite delicate: instead, checking the assumptions necessary for our black box theorem to apply mainly boils down to obtaining some estimates under the invariant measure of the process.\\

Finally we state a convergence result for discrete-time approximations of
$\nabla\phi$ interfaces (with or  without a ``hard wall  constraint'' that imposes positivity of $\phi$) with convex potential in one dimension in Section \ref{sec:nablaphi}, although we do not work out all details. These models have relevance in statistical physics, as describing interface pinning/wetting phenomena in two-dimensions, see for instance the classic reference \cite{Fisher} as well as the monograph \cite{Giacominbook}.\\

Although the three examples presented above are reversible w.r.t.~their invariant measures, we emphasize that our black box result requires neither reversibility, nor the validity of the so-called  ``gradient condition'' \cite{KipLan}.

There is a number of directions that could be investigated in the future. First, our
proof of the convergence of an evolving pair of ordered interfaces could be
extended to cover systems of $k$ ordered interfaces for some arbitrary
but fixed value $k\ge 2$, see Remark \ref{rem:k>2} for more
details.
Second, a very challenging question would be to apply the
black box theorem (or some generalization thereof) to prove the
convergence of the fluctuations of the height function of the
dimer model on the honeycomb graph (or rhombus tilings of the plane), under the ``tower dynamic'' introduced in
\cite{luby2001markov,wilson2004mixing}, to a two-dimensional
stochastic heat equation with additive noise.  Informally speaking,
this tiling dynamics corresponds to a system of $k$ ordered interfaces
with $k$ itself of order $N$. The exact form of the limit stochastic heat equation can be guessed from the hydrodynamic limit equation obtained by the second and third authors in \cite{LT1,LT2}.

We emphasize that the rhombus-tiling tower-dynamic of
    \cite{wilson2004mixing} satisfies a certain ``gradient
    condition'', that comes together with a $L^1$ contraction property \cite[Sec. 3]{LT1}. This  feature allows for instance to prove a hydrodynamic
    limit \cite{LT1,LT2}, that is, the convergence of the height
    profile, under diffusive scaling, to the solution of a non-linear
    parabolic PDE. This is in contrast with the usual rhombus tilings ``single-flip''
    Glauber dynamic, for which no ``gradient
    condition'' or exact contraction property holds, so that even proving that the relaxation time
    scales diffusively is highly non-trivial (see \cite{CMTcmp,LTahp}
    for progress in this direction).

Third, it would be interesting to
apply our black-box result to non-reversible systems, in particular, weakly asymmetric particle systems whose
fluctuations converge to the KPZ equation.

\subsection*{Organization of the paper}
Section \ref{sec:ageneral} contains our ``black box'' theorem that,
under certain general assumptions, allows to obtain the scaling limit
of the fluctuation process of a sequence of Markov chains. The rest of
the work applies the theorem to concrete cases: the zero-range process
on a segment (Section \ref{sec:0}), the pair of reflected
interfaces (Section \ref{Sec:Reflected}) and $1$-dimensional $\nabla\phi$ interfaces with convex potentials (Section \ref{sec:nablaphi}). In the latter case, we do not work out the details.

\subsection*{Acknowledgements} The work of CL and LZ was partially funded by the ANR project Smooth ANR-22-CE40-0017. The work of BL was partially funded by the ANR project DIMERS  ANR-18-CE40-0033. The research of FT was funded in part by the Austrian Science Fund (FWF) 10.55776/F1002.  For open access purposes, the author FT  has applied a CC BY public copyright license to any author-accepted manuscript version arising from this submission. CL is supported by the Institut Universitaire de France.

\section{A black box theorem}
\label{sec:ageneral}
This section presents a general result that ensures convergence in law of a sequence of stationary Markov processes (typically on discrete state-spaces) towards some stationary Markov process (typically on a continuous state-space). It relies on five assumptions: the first assumption concerns the limit process and the four others concern the sequence of discrete processes.

\subsection*{Assumption on the limit}

We consider a càdlàg Markov process $(u(t))_{t\geq 0}$ taking values in a separable\footnote{We do not require $E$ to be a Polish space since completeness is not necessary for our black box theorem to hold. In concrete examples, we will actually consider spaces $E$ that are not complete.} metric space $E$, and we let $\P_x$ be the law of this process starting from $x\in E$. We assume that this Markov process admits a stationary (\emph{not necessarily reversible}) probability measure $m$, which has full support on $E$.\\

For any bounded measurable function $f:E\to \R$ and any $t\ge 0$ we set
$$ P_t f(x) := \E_x[f(u(t))]\;,\quad x\in E\;.$$
By density, $P_t$ can be extended into a bounded operator on $L^2(E,dm)$. We then assume that $(P_t,t\ge 0)$ is a strongly continuous semigroup of bounded operators on $L^2(E,dm)$. The resolvent operators $\cR_\lambda$, $\lambda > 0$ are defined as
$$ \cR_\lambda f := \int_0^\infty e^{-\lambda t} P_t f dt\;,\quad f\in L^2(E,dm)\;.$$
It is well-known that the law of the process $u$ is completely characterized by the knowledge of the action of the operators $\cR_\lambda$, $\lambda > 0$ on a large enough set of functions, see Subsection \ref{Subsec:Reduction} for more details.\\

%
%
%

We can now spell out our assumption, we refer the reader to the discussion below the statement for some explanations.
\begin{assumption}[Characterization of the limit]\label{Ass:Charac}
	There exists a collection of finite complex measures $\Sigma^{\psi}$ on $E$, indexed by $\psi \in \cC$ where $\cC$ is some set of bounded and continuous functions from $E$ to $\C$, such that the following holds. For any $\lambda > 0$ and any bounded Lipschitz function $f:E\to \R$, $\cR_\lambda f$ is the unique bounded Lipschitz function $F:E\to\R$ that satisfies the identity
	\begin{equation}\label{Eq:DirFormResolv}
			-\int_E F(u) \psi(u) \Sigma^{\psi}(du) + \lambda \int_E F(u) \psi(u) m(du) = \int_E f(u) \psi(u) m(du)\;,\quad \forall \psi \in \cC\;.
		\end{equation}
\end{assumption}

	Let us give some explanations on this assumption. An alternative definition of the resolvent operator can be given in terms of the generator $\cL$ of the process: $\cR_\lambda = (\lambda - \cL)^{-1}$. Provided the class $\cC$ is large enough, this leads to the weak characterization of $\cR_\lambda f$ as the unique element such that
	$$\langle (\lambda - \cL) \cR_\lambda f , \psi \rangle = \langle f, \psi \rangle$$
	for all $\psi \in \cC$. By duality, we obtain
	\[
	\forall \psi\in\cC \, \langle  \cR_\lambda f , (\lambda - \cL^*)\psi \rangle = \langle f, \psi \rangle \,.
	\]
	Comparing with equation \ref{Eq:DirFormResolv}, this is consistent if $\cL^* \psi(u) m(du) = \psi(u)  \Sigma^\psi(du)$. However we insist that in Assumption \ref{Ass:Charac} the function $\psi$ does not need to be in the domain of $\cL^*$ (or $\cL$ in the reversible case) so that the identity might only be formal. For example, for the pair of reflected interfaces, one can think of $\cL$ as having a Dirac weight on configurations where the interfaces meet at a point (the only situation where the dynamics sees the reflection). At a rigorous level, this translates as the fact that $ \Sigma^{\psi}$ gives positive mass to these configurations despite them having $m$-mass $0$.\\
	In the reversible case, an important tool to establish a characterization such as Assumption \ref{Ass:Charac} is the general theory of self-adjoint operator and Dirichlet forms. In particular~\cite[Th I.2.8]{MaRockner} says that
	for any $f\in L^2(E,dm)$ and for any $\lambda > 0$, $\cR_\lambda f$ is the unique element of the domain $\cD(\cE)$ of the Dirichlet form such that
	\[\label{Eq:CharactDirichlet}
	\cE(\cR_\lambda f, \psi) + \lambda \int_E \cR_\lambda f(u) \psi(u) m(du) = \int_E f(u)\psi(u) m(du)\;,
	\] 
	for all $\psi$ that lie in (a dense subset of) $\cD(\cE)$. When it is possible to determine $\cD(\cE)$ and to justify an integration by parts in the Dirichlet form, this provides a sufficient condition on the set $\cC$, see the proof of Proposition \ref{prop:Sigma2int} for an example of this method.\\

Let us illustrate the assumption on a concrete example. Take $u$ to be the solution of the stochastic heat equation in dimension $1$
$$ \partial_t u = \frac12 \partial^2_x u + \xi\;,\quad x\in [0,1]\;,\quad t\ge 0\;,$$
endowed with Dirichlet b.c.~where $\xi$ is a space-time white noise. In this example, $E$ can be taken to be $L^1([0,1],dx)$ and $m$ the law of the Brownian bridge on $[0,1]$. The above assumption is then satisfied with $\cC$, the set of all maps $\psi_{\varphi}:E\to\C$ defined by $\psi_\varphi(u) = \exp(i \langle u,\varphi\rangle)$ where $i = \sqrt{-1}$, $\langle \cdot,\cdot\rangle$ is the usual inner product in $L^2([0,1],dx)$ and $\varphi:[0,1]\to\R$ is any $C^2$-function that vanishes at $0$ and $1$. Then
$$ d\Sigma^{\psi_\varphi}(u) = \Big( \frac{i}{2} \langle u,\varphi''\rangle - \frac12 \|\varphi\|_{L^2([0,1],dx)}^2 \Big) m(du)\;.$$
Let us mention that the choice $E=L^1([0,1],dx)$ is relatively arbitrary at this stage. However, for establishing the convergence of approximating processes, this choice will play a role.

\subsection*{Assumptions on the sequence of approximating processes}

We are now given a sequence of Markov processes $u_N$ on some spaces $E_N$. Each process $u_N$ admits an invariant probability measure $m_N$ on $E_N$ which has full support. We let $\mathcal R^N_\lambda$ denote its resolvent operator.

\begin{assumption}[State-space]\label{Ass:State}
	Each $E_N$ is a subset of $E$ and the distance function to $E_N$ is achieved: namely, there exists a map $\pi_N:E\to E_N$ such that for all $x\in E$
	$$ d_E(x , \pi_N(x) ) = \inf_{y\in E_N} d_E( x , y)\;.$$
	In addition, the measure $m_N$ converges weakly towards $m$ as $N\to\infty$.
\end{assumption}

The map $\pi_N$ will be called a projection. Let us emphasize that we do \emph{not} assume that this projection is unique nor that it is continuous. Important examples satisfying this assumption are:\begin{enumerate}
	\item $E_N$ is a finite subset of $E$: in this case existence of $\pi_N$ holds but uniqueness typically fails,
	\item $E$ is a Hilbert space and $E_N$ is a closed convex subset (for instance, a finite-dimensional subspace): existence and uniqueness hold by the Projection Theorem,
	\item $E$ is a subset of a Banach space $B$ (endowed with the same distance/norm), and $E_N \subset E$ is a closed subset of a finite dimensional subspace of $B$: existence holds but uniqueness may fail. For instance, take $E$ to be the subset of all functions of the Banach space $L^1([0,1],dx)$ that are non-negative a.e., and let $E_N$ is the subset of all non-negative functions that are constant on $[k/N,(k+1)/N)$ for every $k\in\{0,\ldots,N-1\}$.
\end{enumerate}

 Below, we identify three conditions on the sequence $u_N$ that ensure that it converges in law towards $u$.

\begin{assumption}[Tightness]\label{Ass:Tight}
	There exists a Polish space $E^{-}$ in which $E$ is continuously embedded, and such that, if each $u_N$ starts from its invariant measure $m_N$, then the sequence of processes $(u_N)_N$ is tight in $\bbD([0,\infty),E^{-})$.
\end{assumption}

This assumption spares us from proving tightness in $\bbD([0,\infty),E)$. In the applications presented later on, $E^-$ will be a Sobolev space with a negative index and the proof of the tightness will follow from the Lyons-Zheng decomposition~\cite{LyonsZheng} for reversible Markov processes.\\

We turn to a regularity assumption on the sequence of resolvents that will allow us to extract converging subsequences. We denote by $\Lip(E_N)$ the set of all Lipschitz functions $f:E_N\to\R$, that is, the set of all maps $f$ such that
$$ [f]_{\Lip(E_N)} := \sup_{x,y \in E_N} \frac{\vert f(x) - f(y) \vert}{d_E(x , y)}<\infty\;.$$

\begin{assumption}[Equicontinuity]\label{Ass:Equi}
	For any given $\lambda > 0$, there exists a constant $C>0$ such that for all $N\ge 1$
$$ [\cR_\lambda^N f]_{\Lip(E_N)} \le C [f]_{\Lip(E_N)}\;,$$
	for all $f\in \Lip(E_N)$.
\end{assumption}

This is the main place where specific properties of the approximating process are used. Indeed in the applications that we will present, this estimate will be derived by probabilistic means, through finding a coupling under which the distance between two copies of the process is a sub-martingale. The choice of the underlying space $E$ plays an important role here: for instance, for some particle systems, the difference of the height functions of two evolving configurations decays stochastically in $L^1$ but not in $L^2$.\\

Finally, we spell out an assumption that can be viewed as an approximation of the characterization of $u$.

\begin{assumption}[Discrete characterization and convergence]\label{Ass:IbPF}
	For any $\psi \in \cC$, there exists a finite complex measure $\Sigma_N^{\psi}$ on $E_N$ such that the following holds. For any $\lambda > 0$ and any $f\in L^2(E_N,dm_N)$, $\cR_\lambda^N f$ satisfies the identity
	\begin{equation}\label{Eq:DirFormResolvN}
		-\int_{E_N} F(u) \psi(u) \Sigma_N^{\psi}(du) + \lambda \int_{E_N} F(u) \psi(u) m_N(du) = \int_{E_N} f(u) \psi(u) m_N(du)\;.
	\end{equation}
	In addition, the finite complex measure $\Sigma_N^{\psi}$ converges weakly to $\Sigma^{\psi}$ as $N\to\infty$.
\end{assumption}

In the particular case where $E_N$ is a finite set and provided $\cC$ is dense in $L^2(E_N,dm_N)$, equation \eqref{Eq:DirFormResolvN} is always verified with $\psi(u) \Sigma_N^{\psi}(du) := \cL^*_N \psi(u) m_N(du)$, where $\cL^*_N$ is the adjoint w.r.t.~$m_N$ of the generator $\cL_N$ of the process. Indeed in that case the discussion after Assumption \ref{Ass:Charac} is always rigorous since it only involves finite dimensional vector spaces. The core of the above assumption is therefore the convergence of the measure $\Sigma_N^{\psi}$ to $\Sigma^{\psi}$, or equivalently of the typically well defined $\cL_N^* \psi(u)$ to the maybe only formal $\cL^* \psi(u)$.

As one can guess, we will pass to the limit on identity \eqref{Eq:DirFormResolvN} along some converging subsequence of resolvents (that can be extracted by Assumption \ref{Ass:Equi}) and we will identify limit points using Assumption \ref{Ass:Charac}.


\bigskip

We now have all the ingredients at hand to state our main result.

\begin{theorem}\label{Th:BlackBox}
	Let $u_N$ start from its invariant measure $m_N$ for every $N\ge 1$. Given Assumptions \eqref{Ass:Charac}, \eqref{Ass:State}, \eqref{Ass:Tight}, \eqref{Ass:Equi} and \eqref{Ass:IbPF}, the sequence $(u_N)_N$ converges in law in $\bbD([0,\infty), E^{-})$ to the process $u$ starting from its invariant measure $m$.
\end{theorem}

Let us mention that in practice the convergence can be strenghtened: if the tightness (Assumption \ref{Ass:Tight}) is established using moment bounds, and if one can establish uniform in $N$ moment bounds under the invariant measure $m_N$, then interpolation inequalities allow to lift the tightness in $\bbD([0,\infty), F)$ where $F$ is a space that interpolates between $E^-$ and $E$. We refer to~\cite[Lemma 5.2]{DebZam} for an application of this idea.\\

The rest of this section is devoted to the proof of this result.

\subsection{Reduction to an identity on the resolvents}\label{Subsec:Reduction}

Let $\bbP_N$ be the law of $u_N$ starting from $m_N$ and let $\bbP$ be the law of $u$ starting from $m$. From the tightness Assumption \ref{Ass:Tight}, all we need to prove is that any limit of a converging subsequence of $(\bbP_N)_N$ matches with $\bbP$. We claim that this is the case provided the following holds.

Assume that for any $n\geq 1$, any bounded functions $f_0,f_1,\ldots,f_n\in\Lip(E^-) \subset \Lip(E)$ and any $\lambda_1,\ldots,\lambda_n >0$, the following convergence holds
\begin{align}\begin{split}\label{Eq:RecurResolv}
	{}&\int_{E_N} f_0(u_0)\cR^N_{\lambda_1} \Big( f_1\cR^N_{\lambda_2} \Big(\ldots f_{n-1} \cR^N_{\lambda_n} f_n\Big)\ldots\Big) (u_0) m_N(du_0)\\
	&\underset{N\rightarrow\infty}{\longrightarrow}
	\int_{E} f_0(u_0)\cR_{\lambda_1} \Big( f_1\cR_{\lambda_2} \Big(\ldots f_{n-1} \cR_{\lambda_n} f_n\Big)\ldots\Big) (u_0) m(du_0)\;.\end{split}
\end{align}

Let us prove this claim. Let $P$ be the limit of a converging subsequence $(\bbP_{N_i})_i$. A monotone class argument yields $P=\bbP$ provided $P(u(t_1)\in A_1;\ldots;u(t_1+\ldots+t_n)\in A_n)=\bbP(u(t_1)\in A_1;\ldots;u(t_1+\ldots +t_n)\in A_n)$ for any collection $A_1,\ldots,A_n$ of closed sets of $E^-$, any collection $t_1,\ldots,t_n \ge 0$ and any integer $n\ge 1$. Since the indicator of any closed set of $E^-$ is the bounded pointwise limit of a sequence of bounded Lipschitz functions on $E^-$, it suffices to show that
\begin{equation}
	P\Big[f_0(u_0)\prod_{j=1}^n f_j(u_{t_1+\ldots+t_j})\Big] = \bbP\Big[f_0(u_0)\prod_{j=1}^n f_j(u_{t_1+\ldots+t_j})\Big]\;,
\end{equation}
holds for all $n$ and bounded $f_0, f_1,\ldots, f_n\in\Lip(E^-)$. By the uniqueness of the Laplace transform of finite measures on $\bbR^n_+$, it boils down to proving\footnote{ Of course $P\Big[f_0(u_0)\prod_{j=1}^n f_j(u_{t_1+\ldots+t_j})\Big] dt_1\ldots dt_n$ is not a finite measure, but $e^{-\sum_{j=1}^n t_j \varepsilon} P\Big[f_0(u_0)\prod_{j=1}^n f_j(u_{t_1+\ldots+t_j})\Big] dt_1\ldots dt_n$ is a finite measure on $\R^n$ for any given $\varepsilon >  0$.}
\begin{align}\begin{split}\label{Eq:IdResolvents}
	&\int_{(0,\infty)^n} e^{-\sum_{j=1}^n \lambda_j t_j} P\Big[f_0(u_0)\prod_{j=1}^n f_j(u_{t_1+\ldots+t_j})\Big] dt_1\ldots dt_n\\
	=& \int_{(0,\infty)^n} e^{-\sum_{j=1}^n \lambda_j t_j} \bbP\Big[f_0(u_0)\prod_{j=1}^n f_j(u_{t_1+\ldots+t_j})\Big] dt_1\ldots dt_n \;,
	\end{split}
\end{align}
for any $n\geq 1$, any $\lambda_1,\ldots,\lambda_n > 0$ and any bounded $f_0, f_1,\ldots, f_n \in \Lip(E^{-})$.

The Markov property applied to the process $u_N$ shows that 
\begin{align}\begin{split}\label{Eq:ResolventsSemigroup}
	&\int_{E_N} f_0(u_0)\cR^N_{\lambda_1} \Big( f_1\cR^N_{\lambda_2} \Big(\ldots f_{n-1} \cR^N_{\lambda_n} f_n\Big)\ldots\Big) (u_0) m_N(du_0) \\
	=& \int_{(0,\infty)^n} e^{-\sum_j \lambda_j t_j} \bbP_N\Big[f_0(u_0)\prod_{j=1}^n f_j(u_{t_1+\ldots+t_j})\Big] dt_1\ldots dt_n \;.
\end{split}\end{align}
The weak convergence of $\bbP_{N_i}$ towards $P$ ensures that the r.h.s.~converges, along the subsequence $N_i$, to
\begin{equation}
	\int_{(0,\infty)^n} e^{-\sum_j \lambda_j t_j} P\Big[f_0(u_0)\prod_{j=1}^n f_j(u_{t_1+\ldots+t_j})\Big] dt_1\ldots dt_n \;.
\end{equation}
On the other hand, the identity \eqref{Eq:RecurResolv} ensures that the l.h.s.~of (\ref{Eq:ResolventsSemigroup}) converges to
\begin{align*}
	{}&\int_E f_0(u_0) \cR_{\lambda_1} \Big( f_1 \cR_{\lambda_2} \Big( \ldots f_{n-1} \cR_{\lambda_n} f_n\Big)\ldots\Big) m(du_0) \\
	=& \int_{(0,\infty)^n} e^{-\sum_j \lambda_j t_j} \bbP\Big[f_0(u_0)\prod_{j=1}^n f_j(u_{t_1+\ldots+t_j})\Big] dt_1\ldots dt_n \;,
\end{align*}
as $N\rightarrow\infty$. Therefore, we have established (\ref{Eq:IdResolvents}), and the equality $\bbP=P$ follows. We have thus reduced the proof of Theorem \ref{Th:BlackBox} to the proof of \eqref{Eq:RecurResolv}.

\subsection{Some technical tools}\label{SubsectionCriterion}

In this paragraph, we collect some technical tools that will eventually allow us to extract converging subsequences from the sequence of discrete resolvents, and to pass to the limit on integrals involving these discrete resolvents against the discrete invariant measures. It turns out that our setting requires some care since continuous maps on $E_N$ cannot be canonically lifted into continuous maps on $E$. This is because in Assumption \ref{Ass:State} we assumed existence of a projection $\pi_N$ from $E$ to $E_N$ without requiring its \emph{continuity}. (It is worth keeping in mind that $\pi_N$ cannot be continuous when $E_N$ is a discrete subset of a continuous set $E$).

Interestingly, our assumptions imply that the amplitude of the possible jumps of the projection $\pi_N$ ``vanish'' as $N\to\infty$. More precisely, for every set $G\subset E$, we denote by $G^\varepsilon$ the open set of all the points in $E$ lying at a distance smaller than $\varepsilon$ from $G$. For all $\varepsilon > 0$, the weak convergence of $m_N$ to $m$  implies that
\begin{equation}\label{Cond:EN}
	E = \varliminf_{N\rightarrow\infty} E_N^\varepsilon \;.
\end{equation}
Indeed, for all $x\in E$ and all $\varepsilon >0$, if we denote by $B(x,\varepsilon)$ the open ball in $E$ centered at $x$ and of radius $\varepsilon$, then we have
\begin{equation*}
	m\big(B(x,\varepsilon)\big) \leq \varliminf_{N\rightarrow\infty} m_N\big(B(x,\varepsilon)\big)\;.
\end{equation*}
Since $m$ has full support on $E$, we deduce that the l.h.s.~is strictly positive so that, for all $N$ large enough $B(x,\varepsilon)\cap E_N \ne \emptyset$, thus ensuring that $x\in E_N^\varepsilon$ for all $N$ large enough. This implies that for all $N$ large enough and all $x,y \in E$, $d( \pi_N(x) , \pi_N(y) ) \le d( x , y ) + 2\varepsilon$ where we write $d$ instead of $d_E$ to lighten notations.

We can now state a modification of Arzel\`a-Ascoli's theorem. Recall from Assumption \ref{Ass:State} that $\pi_N(x)$ is an element of $E_N$ that minimizes the distance from $x$ to $E_N$.

\begin{theorem}[Modified Arzel\`a-Ascoli's theorem]\label{Th:Arzela}
	Let $F_N:E_N\rightarrow\R, N\ge1$ be a collection of uniformly bounded and equicontinuous maps, \textit{i.e.} $\sup_{N\geq 1} \|F_N\|_{\infty}<\infty$ and
	\begin{equation}
		\omega(\delta):= \sup_{N\geq 1}\big\{|F_N(x)-F_N(y)|: x,y\in E_N\;,\; d(x,y)\leq \delta\big\}\;,
	\end{equation}
	goes to $0$ as $\delta \downarrow 0$. Then, there exists a bounded and equicontinuous map $F:E\rightarrow\R$, as well as a subsequence $(N_j)_{j\geq 1}$ such that:
	\begin{enumerate}
		\item For any compact set $J\subset E$ , $\sup_{x\in J\cap E_{N_j}} |F(x)-F_{N_j}(x)| \rightarrow 0$ as $j\rightarrow\infty$,
		\item $F_{N_j}\big(\pi_{N_j} x\big)$ converges to $F(x)$ for all $x \in E$.
	\end{enumerate}
\end{theorem}
\begin{remark}
	The proof shows that the modulus of continuity of $F$ is bounded by $\omega(3\,\cdot)$, and can be adjusted to show that it is actually bounded by $\omega((1+\epsilon)\cdot)$ for any $\epsilon>0$. In particular, if the $F_N$ are uniformly Lipschitz, so is $F$.
\end{remark}
\begin{remark}
	In the case where $E_N$ is a finite dimensional subspace of a Hilbert space $E$, the projection $\pi_N$ is $1$-Lipschitz and then the classical Arzel\`a-Ascoli theorem can be applied to the sequence $F_N(\pi_N\,\cdot)$. In our more general setting, this is unfortunately not the case.
\end{remark}
\begin{proof}
	Let $(x_i)_{i\geq 1} \subset E$ be a dense sequence. By a diagonal argument, there exists a subsequence $N_j \rightarrow\infty$ such that, for every $i\geq 1$, $F_{N_j}(\pi_{N_j} x_i)$ converges to a limit, that we call $F(x_i)$. Necessarily, $\sup_i |F(x_i)| \leq \sup_N \|F_N\|_{\infty}<\infty$. Fix $i\ne \ell \geq 1$. From (\ref{Cond:EN}), we deduce that both $d(x_i,\pi_{N_j} x_i)$ and $d(x_\ell,\pi_{N_j} x_\ell)$ are smaller than $d(x_i,x_\ell)/2$ as soon as $j$ is large enough. Then, we have
	\begin{equation}\label{Eq:UnifDens}
		|F(x_i) - F(x_\ell)| = \varlimsup_{j\rightarrow\infty} |F_{N_j}(\pi_{N_j} x_i) - F_{N_j}(\pi_{N_j} x_\ell)|\leq \omega\big(2\, d(x_i,x_{\ell})\big)\;.\quad
	\end{equation}
	Let $x\in E$. For each $k\geq 1$, let $i_k\geq 1$ be such that $x \in B(x_{i_k},k^{-1})$. Using (\ref{Eq:UnifDens}), it is then elementary to check that $F(x_{i_k}),k\geq 1$ is a Cauchy sequence. Let $F(x)$ be its limit: using \eqref{Eq:UnifDens}, it is straightfoward to check that this definition does not depend on the choice of the sequence $i_k$, and that it matches with the map $F$ already defined on $\{x_i, i\ge 1\}$. Using (\ref{Eq:UnifDens}) once again, it is simple to obtain the bound $|F(x)-F(y)| \leq \omega\big(3\,d(x,y)\big)$ for all $x,y\in E$. The function $F$ is therefore equicontinuous on $E$.
	
	Let $J$ be a compact set of $E$. For any $k\geq 1$, there exists $p_k\geq 1$ such that $\cup_{i\leq p_k} B(x_{i},k^{-1})$ covers $J$. By \eqref{Cond:EN}, taking $j=j(k)$ large enough, for all $i\in\{1,\ldots,p_k\}$ we have $x_i \in E_{N_j}^{1/k}$ together with
	\begin{equation}
		|F_{N_j}(\pi_{N_j} x_{i})-F(x_{i})| \leq \frac{1}{k} \;.
	\end{equation}
	Let $x\in J$. There exists $i_k \leq p_k$ such that $d(x,x_{i_k})\leq k^{-1}$. Then, for all $j$ large enough we get
	\begin{align*}
		|F(x)-F_{N_j}(\pi_{N_j} x)| &\leq |F(x)-F(x_{i_k})| + |F(x_{i_k})-F_{N_j}(\pi_{N_j} x_{i_k})|\\
		&\quad+ |F_{N_j}(\pi_{N_j} x_{i_k})-F_{N_j}(\pi_{N_j} x)|\\
		&\leq \omega\big(3 k^{-1}\big) + k^{-1} + \omega\big(2k^{-1}+d(x,\pi_{N_j} x)\big)\;.
	\end{align*}
	If $x\in J\cap E_{N_j}$, then $x=\pi_{N_j} x$ and the first assertion of the theorem follows. Regarding the second assertion, fix $x\in E$ and observe that we can apply the previous arguments to $J=\{x\}$. By (\ref{Cond:EN}), we know that for all $j$ large enough $d(x,\pi_{N_j} x)\le k^{-1}$ and therefore the previous bound ensures that $F_{N_j}(\pi_{N_j} x)$ converges to $F(x)$. This ends the proof.
\end{proof}
We conclude this section by a criteria for ``diagonal'' convergence where both the measure and the integrand depend on $N$. The proof is very close to~\cite[Lemma 1]{Zambotti2001}.
\begin{lemma}\label{Lemma:EquiContCV}
	Let $\psi:E\to \mathbb C$ be a bounded and continuous function. Let $F_N:E_N\rightarrow \R$, $N\geq 1$ be a sequence of uniformly bounded and equicontinuous maps in the sense of Proposition \ref{Th:Arzela}, and assume that $F_N(\pi_N\cdot)$ converges $m$-almost everywhere to some bounded continuous map $F:E\rightarrow\R$. Suppose that $\rho_N$ is a sequence of complex finite measures on each $E_N$, that converges weakly to some complex finite measure $\rho$ on $E$. Then $\int F_N \psi \,d\rho_N \rightarrow \int F \psi \,d\rho$.
\end{lemma}
\begin{proof}
	Without loss of generality, we can assume that $0\leq F_N \leq 1$, $0\leq \psi \leq 1$ and that the measures $\rho_N$ are real-valued and non-negative. Let ${N}_j,j\geq 1$ be a sequence going to $\infty$. By Theorem \ref{Th:Arzela}, we can extract a subsubsequence $F_{N_{i_j}}$ such that the two assertions of the theorem are satisfied with the limit being necessarily the function $F$ by continuity. The tightness of the sequence $\rho_{N_{i_j}}$ implies that for every $n\geq 1$, there exists a compact set $J_n$ such that $\varlimsup_j \rho_{N_{i_j}}\big(J_n^c\big) \leq n^{-1}$. Using the fact that $\rho_{N_{i_j}}$ is supported by $E_{N_{i_j}}$, we write
	\begin{equation}
		\int_E F_{N_{i_j}} \psi d\rho_{N_{i_j}} \leq \rho_{N_{i_j}}\big(J_n^c\big) + \int_{J_n} \big(F_{N_{i_j}}(\pi_{N_{i_j}}\cdot)-F(\cdot)\big) \psi d\rho_{N_{i_j}} + \int_{J_n} F \psi d\rho_{N_{i_j}}\;.
	\end{equation}
	The first assertion of the theorem and the boundedness of $\psi$ imply that \[\varlimsup_{j}\int_{J_n} |F_{N_{i_j}}(\pi_{N_{i_j}}\cdot)-F(\cdot)| \psi d\rho_{N_{i_j}} = 0.\] The weak convergence of $F\psi\rho_N$ towards $F\psi\rho$ ensures that \[\varlimsup_{j} \int_{J_n} F \psi d\rho_{N_{i_j}} \leq \int_{J_n} F \psi d\rho.\] Therefore, we get
	\begin{equation}
		\varlimsup_{j} \int_E F_{N_{i_j}} \psi d\rho_{N_{i_j}} \leq \frac{1}{n} + \int_E F \psi d\rho\;.
	\end{equation}
	Passing to the limit on $n$, we obtain $\varlimsup_{j} \int F_{N_{i_j}} \psi d\rho_{N_{i_j}} \leq \int F \psi d\rho$. Upon replacing $F_N$ by $1-F_N$, one obtains $\varliminf_{j} \int F_{N_{i_j}} \psi d\rho_{N_{i_j}} \geq \int F \psi d\rho$. This ends the proof.
\end{proof}

\subsection{Convergence of the resolvents}

We come back to our original goal: the proof of \eqref{Eq:RecurResolv}. It boils down to the following lemma.

\begin{lemma}\label{Lemma:CVResolv}
	Let $g_N:E_N\rightarrow\R$, $N\geq 1$ be a sequence of uniformly bounded and uniformly Lipschitz maps, that converges $m$-a.e.~to some bounded Lipschitz function $g:E\rightarrow\R$. Then, for all $\lambda > 0$, the map $\cR_\lambda^N g_N(\pi_N\cdot)$ converges pointwise to $\cR_\lambda g(\cdot)$ as $N\rightarrow\infty$.
\end{lemma}
\begin{proof}
	Using Assumption \ref{Ass:Equi}, we know that the sequence $F_N(\cdot)=\cR_\lambda^N g_N(\cdot)$, $N\geq 1$, is uniformly bounded and equicontinuous in the sense of Theorem \ref{Th:Arzela}. Furthermore Assumption \ref{Ass:IbPF} states that for all $\psi \in \cC$
	\begin{equation*}
		-\int_{E_N} F_N(u) \psi(u) \Sigma_N^{\psi}(du) + \lambda \int_{E_N} F_N(u) \psi m_N(du) = \int_{E_N} g_N(u) \psi(u)\, m_N(du)\;.
    \end{equation*}
	Let $F_{N_i}, N_i\geq 1$ be an arbitrary subsequence. By Theorem \ref{Th:Arzela}, there exists a bounded Lipschitz map $F$ on $E$ and a subsubsequence $F_{N_{i_j}}$ converging to $F$ in the sense of the statement. The weak convergences of $\Sigma_N^\psi$ and $m_N$ to $\Sigma^\psi$ and $m$ (guaranteed by Assumptions \ref{Ass:IbPF} and \ref{Ass:State}), together with the regularity properties of the integrands are sufficient to apply Lemma \ref{Lemma:EquiContCV} along the subsubsequence $N_{i_j}$ and deduce that
	\begin{equation*}
		-\int_{E} F(u) \psi(u) \Sigma^{\psi}(du) + \lambda \int_{E} F(u) \psi(u) m(du) = \int_{E} g(u) \psi(u)\, m(du)\;.
	\end{equation*}
	Since this holds for all $\psi \in \cC$, Assumption \ref{Ass:Charac} ensures that $F=\cR_\lambda g$. Since the subsequence $N_i$ was arbitrary, we can conclude.
\end{proof}
A recursion based on Lemmas \ref{Lemma:EquiContCV} and \ref{Lemma:CVResolv} then yields \eqref{Eq:RecurResolv}.

\section{The zero-range process}	
\label{sec:0} For pedagogical reasons, we show first how
  our method works when applied to the one-dimensional zero-range
  process. In this case, the convergence of the fluctuation process is
  already known (see \cite[Chap. 11]{KipLan} and references therein),
  which allows the reader to compare our new approach to the classical
  one. Note that we bypass entirely the need to prove the so-called
  ``Boltzmann-Gibbs principle'' (\cite[Sec. 11.1]{KipLan}). The equicontinuity of the resolvents will follow
from a model dependent but very soft argument, while the ``replacement'' part of the
proof, Assumption 2.5, uses technical but routine local CLT results.

\subsection{Model and notations}
	
A configuration $n$ is a map from the lattice $\{1,\ldots,N\}$ to $\N$ and we interpret $n(i)$ as a number of particles at site $i$. The jump rate of the particles is given by a function $\tau : \N \to \R_+$ on which we assume: (1) $\tau(0) = 0$, (2) $\tau$ is a non-decreasing function, (3) $\tau$ grows at most polynomially at infinity. (The second hypothesis implies attractiveness).

The dynamics is as follows: at rate $N^2\tau(n(i))$ a Poisson clock rings at site $i$ and when this happens a particle jumps right or left by one step, each with probability $1/2$, except for $i=1$ (resp.~$i=N$) in which case the jump to the left (resp.~to the right) is not allowed. In other words, at rate $N^2\tau(n(i))/2$, $(n(i), n(i+1)) \to (n(i) -1, n(i+1) + 1)$ whenever $i\in \{1,\ldots,N-1\}$ and at the same rate $(n(i), n(i-1)) \to (n(i) -1, n(i-1) + 1)$ whenever $i\in \{2,\ldots,N\}$.\\ 

Define the probability measure
$$ \nu_a(k) := \frac1{Z_a} \frac{a^k}{\prod_{i=1}^k \tau(i)}\;,\quad k\in\N\;,\quad Z_a := \sum_{j\ge 0} \frac{a^j}{\prod_{i=1}^j \tau(i)}\;.$$
Let $a^* := \sup\{a>0: Z_a < \infty\}$. Depending on the growth rate of the function $\tau$, this quantity is either finite or infinite. We introduce $R(a) := \sum_k k \nu_a(k)$, the first moment of $\nu_a$, and $R^* := \lim_{a\uparrow a^*} R(a)$. We then let $\Phi:(0,R^*) \to (0,a^*)$ be the reciprocal of $R$: for any parameter $\overline{n} \in (0,R^*)$, the measure $\nu_{\Phi(\overline{n})}$ has mean $\overline{n}$. Given some $\overline{n} \in (0,R^*)$, the above dynamics, restricted to the set $\Omega_{N,\overline{n}}$ of configurations with $\lfloor N \overline{n}\rfloor$ particles, admits a unique invariant (and reversible) probability measure
$$ \pi_{N,\overline{n}}(n) = \frac{C(\overline{n},N)}{\prod_{i=1}^N \prod_{j=1}^{n(i)} \tau(j)}\;,\quad \forall n=(n(1),\ldots,n(N)) \in \Omega_{N,\overline{n}}\;.$$
This measure is nothing but the product measure of the $\nu_a$ conditioned to $\Omega_{N,\overline{n}}$, for any arbitrary parameter $a\in (0,a^*)$. Let us point out however that a ``natural'' choice of parameter $a$ would be $\Phi(\overline{n})$.\\
From now on, $\overline{n} \in (0,R^*)$ is fixed but arbitrary.\\


Given an evolving configuration $(n_N(t,i), t\ge 0, i\in\{1,\ldots,N\})$, we introduce the evolving height function as follows
\begin{equation}\label{Eq:ntou}
	u_N(t,x) := \frac1{\sqrt{N}} \sum_{k=1}^{\lfloor x N \rfloor}\Big( n_N(t,k) - \overline{n} \Big)\;,\quad x\in [0,1]\;.
\end{equation}
It turns out that it is convenient to work on the Banach space $E=L^1([0,1],dx)$.\\

Since the mapping $n_N \to u_N$ is bijective, $(u_N(t), t\ge 0)$ is a Markov process taking values in $E$. We let $\mathcal{L}_N$ be its generator, $m_N$ its invariant probability measure, and $\mathcal{E}_N$ the associated Dirichlet form under the invariant measure $m_N$. Note that $u_N$ actually takes values in a finite set $E_N \subset E$.\\

For later purposes, we derive some scaling limits under the invariant measure. These limits depend on a couple of parameters defined as follows. Set $a:= \Phi(\overline{n})$ and
\begin{align*}
	\overline{\tau} &= \EE_{\nu_{a}}[\tau(n)] = \sum_k \tau(k) \nu_{a}(k)\;,\quad &&\rho = \Cov_{\nu_{a}}(n,\tau(n)) =  \sum_k (k-\overline{n})(\tau(k)-\overline{\tau}) \nu_{a}(k)\;,\\
	\alpha &= \Var_{\nu_{a}}(n) = \sum_k (k-\overline{n})^2 \nu_{a}(k)\;,\quad &&\gamma = \Var_{\nu_{a}}(\tau(n)) = \sum_k (\tau(k)-\overline{\tau})^2 \nu_{a}(k)\;.
\end{align*}
Elementary computations show that
\[
\overline{\tau} = \rho = a\;.
\]
Given a configuration $n_N$ we introduce the height function associated to the field of jump rates:
$$ q_N(x) := \frac1{\sqrt{N}} \sum_{k=1}^{\lfloor x N \rfloor}( \tau(n_N(k)) - \overline{\tau})\;,\quad x\in [0,1]\;.$$

\begin{lemma}\label{Lemma:Gaussian} The pair $(u_N,q_N)$, under the probability measure $m_N$, converges in law to a centered Gaussian process $(\beta(x),\zeta(x))_{0\le x \le 1}$ whose covariance is characterized by
		\begin{align*}
		\Cov(\beta(s),\beta(t)) &= \alpha \min(s,t)(1-\max(s,t)) \;,\quad \Cov(\zeta(s),\zeta(t)) = \gamma \min(s,t) - \frac{\rho^2}{\alpha} st \;,\\
		  		\Cov(\beta(s),\zeta(t)) &= \rho \min(s,t)(1-\max(s,t))\;.
		\end{align*}
		In addition the following bounds hold: for all $p\ge 1$
		$$ \sup_N m_N[\sup_{x\in [0,1]} \vert u_N(x) \vert^p] < \infty\;,\quad \sup_N m_N[\sup_{x\in [0,1]} \vert q_N(x) \vert^p] < \infty\;, \quad \sup_N m_N[\tau(n(1))^p] < \infty\;.$$
\end{lemma}
This result implies that $\frac1{\sqrt{\alpha}} \beta$ is a standard Brownian bridge on $[0,1]$, and that $\zeta - \frac{\rho}{\alpha} \beta$ is a centered Gaussian process independent of $\beta$. 
\begin{proof}
	To prove the convergence of $({u}_N,q_N)$ under $m_N$, one first considers the product law $\mu_N:=\otimes_{k=1}^N \nu_{a}$ with $a=\Phi(\overline{n})$. It is standard to prove that under\footnote{with minor abuse of notation, we let $\mu_N$ denote also the law of $(u_N,q_N)$ when $n_N$ has law $\mu_N$} $\mu_N$, the pair $({u}_N,q_N)$ converges towards a centered Gaussian process $(B,Z)$ with covariance
	\begin{align*}
		\Cov(B(s),B(t)) &= \alpha \min(s,t) \;,\quad \Cov(Z(s),Z(t)) = \gamma \min(s,t) \;,\\
		\Cov(B(s),Z(t)) &= \rho \min(s,t)\;.
	\end{align*}
	($B$ and $Z$ are scaled and correlated Brownian motions). At this point, we observe that $Z = \frac{\rho}{\alpha} B + Z'$ where $Z'$ is a centered Gaussian process independent of $B$. Consequently if we let $(\beta,\zeta)$ be the centered Gaussian process obtained by conditioning $(B,Z)$ to $B(1) = 0$, then $\beta$ is a scaled Brownian Bridge and $\zeta = \frac{\rho}{\alpha} \beta + Z'$, with $\beta$ independent of $Z'$. The covariance structure of the statement follows.\\
	We now concentrate on the proof of the convergence of ${u}_N$ under $m_N$ towards $\beta$: the extension to the convergence of the pair $({u}_N,q_N)$ under $m_N$ is relatively straightforward. By~\cite[Th.1 Chap VII]{Petrov} a Local Limit Theorem holds: for all $x\in (0,1]$, if we let $k_N := \lfloor xN\rfloor$ then
	\begin{eqnarray}
		\label{eq:LLT}
		\sup_{y \in \frac{\bbN - k_N\overline{n}}{\sqrt{N}}} \Big\vert \sqrt N\mu_N(u_N(k_N/N) = y) -  g_{{\alpha x}}(y) \Big\vert \to 0\;,\quad N\to\infty\;,          
	\end{eqnarray}
	where $g_{\sigma^2}$ denotes the density of the centered Gaussian of variance $\sigma^2$. The independence and stationarity of the increments allows to extend this LLT to finite dimensional marginals.\\
	Following an observation made earlier in this subsection and given the definition of ${u}_N$, we have
	$$ m_N = \mu_N (\cdot \mid {u}_N(1) = \delta_N)\;,\quad \delta_N := \frac1{\sqrt N}(\lfloor N\bar n\rfloor - N\bar n)\;.$$
	This allows to deduce the following uniform in $N$ absolute continuity result. There exists a constant $C>0$ such that for all $N$ large enough and all non-negative and measurable map $f:\R^{k_N} \to \R_+$
	\begin{equation}\label{Eq:UnifAbsCont}
		m_N\Big[ f({u}_N(1/N),\ldots, {u}_N(k_N/N))\Big] \le C \mu_N\Big[ f({u}_N(1/N),\ldots, {u}_N(k_N/N))\Big]
	\end{equation}
	where $k_N := \lceil N/2\rceil$.
	Indeed, $C$ can simply be taken to be the supremum over all $N$ large enough of
	$$ \sup_{y \in\frac{\bbN - k_N \overline{n}}{\sqrt{N}}} \frac{\mu_N({u}_N(1)-{u}_N(k_N/N) = -y)}{\mu_N({u}_N(1)=0)} =\sup_{y \in\frac{\bbN - k_N \overline{n}}{\sqrt{N}}} \frac{\mu_N({u}_N(1-k_N/N) = -y)}{\mu_N({u}_N(1)=0)}\;,$$
	and this quantity is finite thanks to \eqref{eq:LLT}.\\
	Let us now establish that the sequence $(u_N)_N$ is tight. Actually, our proof will also establish the moment bounds of the statement on the expectation of the suprema of $u_N$. To that end, we establish an estimate on the H\"older semi-norm of ${u}_N$: for any $\beta \in (0,1/2)$ and $p\ge 1$
	\begin{eqnarray}
		\label{eq:tedioso}
		\limsup_{N\ge 1} \sup_{k \ne \ell \in \{1,\ldots,N\}}m_N\left[\left(\frac{\vert {u}_N(k/N) - {u}_N(\ell/N)\vert}{\vert \frac{k}{N} - \frac{\ell}{N}\vert^{\beta}}\right)^p\right] < \infty\;. 
	\end{eqnarray}
	By symmetry and using the triangle inequality, we can assume that $k,\ell\le N/2$. From the uniform absolute continuity \eqref{Eq:UnifAbsCont}, it suffices to establish \eqref{eq:tedioso} under $\mu_N$ rather than $m_N$. This estimate follows from standard techniques, using the fact that under $\nu_a$, $n-\bar{n}$ admits moments of all orders.\\
	Let us now deduce the convergence of the finite dimensional marginals of ${u}_N$ under $m_N$ towards those of $\beta$. For simplicity, let us concentrate on one-dimensional marginals. Fix $a<b$ and $x\in (0,1)$, and set $k_N := \lfloor xN\rfloor$. We compute
	\begin{align*}
		m_N\big( a < {u}_N(k_N/N) < b \big) &= \sum_{y\in \frac{\bbN - k_N \overline{n}}{\sqrt{N}} : a < y < b} m_N({u}_N(k_N/N) = y)\\
		&= \sum_{y\in \frac{\bbN - k_N \overline{n}}{\sqrt{N}} : a < y < b}  \frac{\mu_N({u}_N(k_N/N) = y)\mu_N({u}_N(1-k_N/N) = -y+\delta_N)}{\mu_N({u}_N(1) = \delta_N)}\;.
	\end{align*}
	Now observe that
	$$\inf_{y \in (a,b)} g_{{\alpha x}}(y) > 0\;,\quad \inf_{y \in (a,b)} g_{{\alpha (1-x)}}(-y+\delta_N) > 0\;,\quad g_{{\alpha}}(\delta_N) \to g_\alpha(0) > 0\;.$$
	Consequently the LLT \eqref{eq:LLT} allows to write (with $o(1)$ being uniform over all $y$ in the sum)
	\begin{align*}
		m_N\big( a < {u}_N(k_N/N) < b \big) &= \sum_{y\in \frac{\bbN - k_N \overline{n}}{\sqrt{N}} : a < y < b}  \frac1{\sqrt{N}}\frac{g_{\alpha x}(y) g_{\alpha (1-x)}(-y+\delta_N)}{g_{{\alpha}}(\delta_N)}(1+o(1))\\
		&\to \int_{a<y<b} \frac{g_{\alpha x}(y) g_{\alpha (1-x)}(-y)}{g_{{\alpha}}(0)} dy = \P(a < \beta(x) < b)\;.
	\end{align*}
	Combined with the bound \eqref{eq:tedioso} already established, this is sufficient to deduce that the law of ${u}_N(x)$ under $m_N$ converges to that of $\beta(x)$.\\
	Let us finally establish the bound on the moments of $\tau(n(1))$. From the absolute continuity \eqref{Eq:UnifAbsCont}, it is sufficient to prove
	$$ \limsup_N \nu_a[\tau(n(1))^p] < \infty\;,$$
	and this follows from the fact that $\tau$ grows at most polynomially at infinity.    
\end{proof}
Note that Lemma \ref{Lemma:Gaussian} implies that Assumption \ref{Ass:State} holds (recall item (3) in the discussion after the formulation of that assumption).

\subsection{The SPDE}\label{Subsec:SPDEZR}

Consider the solution $u$ of
\begin{equation}\label{Eq:SPDEZRP}
	\begin{cases} \partial_t u = c\partial_x^2 u + \sigma {\xi}\;,\quad x\in [0,1]\;,\quad t\ge 0\;,\\
		u(t=0,\cdot) = u_0(\cdot)\;,
		\end{cases}
\end{equation}
endowed with Dirichlet b.c., where $c,\sigma>0$ are some constants, ${\xi}$ is a gaussian space-time white noise and $u_0$ is some initial condition in $E=L^1([0,1],dx)$. (We refer to Subsection \ref{Sec:SPDEPair} for the precise meaning given to such an equation).\\

The process $u$ admits a unique invariant and reversible measure $m_{c,\sigma}$ which is the law of a scaled Brownian bridge (pinned at $0$ at $x \in\{0,1\}$). More precisely, let $Y(x),x\in [0,1]$ be a standard Brownian bridge then $m_{c,\sigma}$ is the law of $\sqrt{\frac{\sigma^2}{2c}} Y$.\\

Since $u$ is a Feller process, it is associated to a strongly continuous semigroup $P_t$ in $L^2(E,m)$, the latter being the space of all maps $F:E \to\R$ that are square integrable against $m$. The general theory ensures that $P_t$ gives rise to a generator $\cL$, a Dirichlet form $\cE$ and resolvents $\cR_\lambda$, $\lambda > 0$.\\

Our next result shows that Assumption \ref{Ass:Charac} is satisfied for the process $u(t,\cdot)$. The class $\cC$ is taken to be the set of all functions $\psi_\varphi(u) := \exp(i \langle u,\varphi\rangle)$, where $\varphi:[0,1]\to\R$ is a $C^2$ function compactly supported in $(0,1)$, and where $\langle \cdot , \cdot \rangle$ is the inner product in $L^2([0,1],dx)$.

\begin{proposition}
  Let $f\in L^2(E,dm)$ be a Lipschitz function from $E$ to $\R$.
  For any $\lambda > 0$, $\cR_\lambda f$ is the unique element of $L^2(E,dm)$ such that for any $C^2$ function $\varphi:[0,1]\to\R$ with compact support in $(0,1)$, the following holds:
	\begin{equation}\label{Eq:CharactResolZRP}
		\int \cR_\lambda f(u) \psi_\varphi(u) \Big( \lambda +\frac{\sigma^2}{2}  \| \varphi\|_{L^2}^2 - ic \langle u,\varphi''\rangle \Big) m(du) = \int f(u) \psi_\varphi(u) m(du)\;.
	\end{equation}
	As a consequence Assumption \ref{Ass:Charac} is satisfied with the class $\cC$ and the measures
	\begin{equation}\label{Eq:SigmaZRP}
		\Sigma^{\psi_\varphi}(du) = \Big( ic \langle u,\varphi''\rangle -\frac{\sigma^2}{2}  \| \varphi\|_{L^2}^2 \Big) dm(u)\;,\quad \psi_\varphi \in \cC\;.
	\end{equation}
\end{proposition}
Such a result is certainly not new, but for the sake of completeness we provide a proof. Let us emphasize that the uniqueness part relies on elementary facts on finite-dimensional Ornstein-Uhlenbeck processes.
\begin{proof}
	\textit{Step 1: extension to more general $\varphi$.}
	Assume that some map $F\in L^2(E,dm)$ satisfies for all $C^2$ function $\varphi$ with compact support in $(0,1)$
	\begin{equation}\label{Eq:CharactZRPNew}
		\int F(u) \psi_\varphi(u) \Big( \lambda +\frac{\sigma^2}{2}  \| \varphi\|_{L^2}^2 - ic \langle u,\varphi''\rangle \Big) m(du) = \int f(u) \psi_\varphi(u) m(du)\;.
	\end{equation}
	Let us show that it remains true for all $C^2$ function $\varphi$ that is not necessarily compactly supported in $(0,1)$ but only vanishes at the boundaries. First of all, it is not hard to prove the following moment bound on the Brownian bridge: for any given $p\ge 1$
	$$ m\Big[\sup_{x\in [0,2/n] \cup [1-2/n,1]} |u(x)|^p\Big] \lesssim (1/n)^{p/2}\;.$$
	uniformly over all $n\ge 1$. Now let $\chi_n:[0,1]\to [0,1]$ be a smooth function that equals $1$ on $[2/n,1-2/n]$ and $0$ outside $[1/n,1-1/n]$ and such that there exists $C>0$ such that for all $n\ge 1$
	$$ \sup_{x} |\chi_n'(x)| \le Cn\;,\quad \sup_x | \chi_n''(x)| \le Cn^2\;.$$
	Set $\varphi_n(x) := \varphi(x)\chi_n(x)$. Then \eqref{Eq:CharactZRPNew} evaluated at $\varphi_n$ converges to the same equation with $\varphi$. The only delicate term comes is $\langle u,\varphi\rangle$, let us provide the details. Since $\varphi-\varphi_n$ vanishes outside $I_n := [0,2/n] \cup [1-2/n,1]$ we find
	\begin{align*}
		&\int \big| F(u) \psi_\varphi(u) \langle u,(\varphi''-\varphi_n'')\rangle \big| m(du)\\
		&\le \Big(\int F(u)^2 m(du)\Big)^{1/2} m\Big[\sup_{x\in I_n} |u(x)|^2\Big]^{1/2} \int_{[0,2/n]\cup [1-2/n,1]} \big| \varphi''(1-\chi_n) + 2 \varphi'\chi_n' + \varphi \chi_n''\big|\;.
	\end{align*}
The product of the first two terms go to $0$ as $n\to\infty$. Since $\varphi$ vanishes at $0$ and $1$, it is of order $1/n$ on $I_n$, and the integral can be bounded by a term of order $1$ uniformly over all $n\ge 1$.\\
	\textit{Step 2: existence.}
	The general theory~\cite[Th I.2.8]{MaRockner} ensures that $\cR_\lambda f$ is the only $F\in\cD(\cE)$ satisfying
	\begin{equation}\label{eq:page16}
		\cE(F, g) +\lambda \int F(u) g(u) m(du) = \int f(u) g(u) m(du)\;,
	\end{equation}
	for all $g \in \cD(\cE)$, the latter being the domain of the Dirichlet form. 
	The general theory again ensures that for any $g\in\cD(\cL)$, the domain of the generator, it holds
	$$ \cE(\cR_\lambda f, g) = -\int \cR_\lambda f(u) \cL g(u) m(du)\;.$$
	Assume for now that $\psi_\varphi$ belongs to $\cD(\cL)$, and that
	$$ \cL \psi_\varphi(u) =  \psi_\varphi(u) \Big(  ic \langle u,\varphi''\rangle  -  \frac{\sigma^2}{2} \| \varphi\|_{L^2}^2\Big)\;.$$
	Combining the last three identities, and using the general fact that $\cD(\cL) \subset \cD(\cE)$, we deduce that $\cR_\lambda f$ satisfies \eqref{Eq:CharactResolZRP}.\\
	\textit{Step 3: computation of $\cL \psi_\varphi$.}
	Let us show that $\psi_\varphi$ belongs to $\cD(\cL)$. Instead of identifying explicitly the domain of the generator, we prove that the following convergence holds in $L^2(E,dm)$
	$$\frac{P_t \psi_\varphi - \psi_\varphi}{t}(u) \to  \psi_\varphi(u) \Big(  ic \langle u,\varphi''\rangle  -  \frac{\sigma^2}{2} \| \varphi\|_{L^2}^2\Big)\;,\quad t\downarrow 0\;.$$
	Recall that the Markov process $u(t,\cdot)$ can be built as the solution of \eqref{Eq:SPDEZRP}. From this evolution equation, we deduce that
	$$ d \langle u(t,\cdot),\varphi\rangle = c \langle u(t,\cdot),\varphi''\rangle dt + \sigma \| \varphi \|_{L^2} dB(t)\;,$$
	where $B$ is a Brownian motion. Since $P_t \psi_\varphi(u_0) = \E[e^{i \langle u(t,\cdot),\varphi\rangle} \mid u(t,0)= u_0 ]$, an explicit computation yields the desired result.\\
	\textit{Step 4: uniqueness.}
	Let us show that this is the unique $L^2(E,dm)$ map that satisfies this identity. Given the first step above, it suffices to prove that if $F\in L^2(E,dm)$ satisfies
	$$ \int F(u) \psi_\varphi(u) \Big( \lambda + \frac{\sigma^2}{2} \| \varphi\|_{L^2}^2 - ic \langle u,\varphi''\rangle \Big) m(du) = 0\;,$$
	for all $C^2$ function $\varphi$ that vanishes at the boundaries, then $F=0$.\\
	Let $F$ be such a function. Let $e_\mu$ be a normalized eigenfunction of the Dirichlet Laplacian on $[0,1]$ associated with some eigenvalue $-\mu \le 0$. For $t\in\R$, set $\varphi := t e_\mu$. Then
	$$ \psi_\varphi(u)(\lambda + \frac{\sigma^2}{2} \| \varphi\|_{L^2}^2 - ic \langle u,\varphi''\rangle) = e^{it \langle u,e_\mu\rangle} (\lambda +t^2 \frac{\sigma^2}{2} + i \mu c t \langle u,e_\mu\rangle)\;.$$
	Writing $x := \langle u,e_\mu\rangle$ and $g(x) := e^{itx}$, we observe that the last expression coincides with $(\lambda - \cL^{\tiny OU})g(x)$, where $\cL^{\tiny OU}$ is the generator of the 1-d Ornstein-Uhlenbeck process
	$$ dX(t) = -\mu c X(t) dt + \sigma dB(t)\;.$$
	Its invariant probability measure $m^{\tiny OU}$ is the pushforward of $m$ through $u\mapsto \langle u,e_\mu\rangle$. Let $\cF_\mu$ be the sigma-field on $E$ generated by this last map and let $F_\mu:\R\to\R$ be the measurable map s.t.~$m[F \mid \cF_\mu](u) = F_\mu(\langle u,e_\mu\rangle)$. We have shown that
	\begin{align*}
		\int_\R F_\mu(x) (\lambda - \cL^{\tiny OU})g(x) dm^{\tiny OU}(x) = \int_E F(u) \psi_\varphi(u) \Big( \lambda + \frac{\sigma^2}{2} \| \varphi\|_{L^2}^2 - ic \langle u,\varphi''\rangle \Big) m(du) =0\;.
	\end{align*}
	As a consequence
	$$ \int F_\mu(x) (\lambda - \cL^{\tiny OU})h(x) m^{\tiny OU}(dx) =0\;,$$
	for all $h$ in the vector space spanned by $\{x\mapsto e^{itx}: t\in \R\}$. By approximation arguments, for any $f\in L^2(\R,dm^{\tiny OU})$ there exists a sequence $h_n$ in the latter vector space s.t.~$\|(\lambda - \cL^{\tiny OU}) h_n - f\|_{L^2(\R,dm^{\tiny OU})}$ goes to $0$ as $n\to\infty$ and this suffices to deduce that
	$$ \int F_\mu(x) f(x) m^{\tiny OU}(dx) =0\;.$$
	Consequently $F_\mu = 0$ $m^{\tiny OU}$-a.e., and in turn $m[F \mid \cF_\mu]= 0$ $m$-a.e.\\
	The very same arguments applies if we consider finitely many $e_{\mu_i}$ simultaneously. More precisely, if we enumerate the eigenvalues of the Dirichlet Laplacian $0 > -\mu_1 > -\mu_2 > \ldots$ and if we let $\cF_n$ be the sigma-field on $E$ generated by the maps $u\mapsto \langle u,e_{\mu_1}\rangle$, $\ldots$, $u\mapsto \langle u,e_{\mu_n}\rangle$, then $m[F \mid \cF_n]= 0$ $m$-a.e.~for any $n\ge 1$. As a consequence $m[F \mid \cF_\infty]= 0$ $m$-a.e., where $\cF_\infty$ is the smallest $\sigma$-field generated by all the $\cF_n$, $n\ge 1$ augmented with all $m$-null sets. Now we observe that $L^2 \subset L^1 = E$ and that $m(L^2) = 1$. Since the $(e_{\mu_n})_{n\ge 1}$ is a Hilbert basis of $L^2$, we deduce that for $m$-a.e.~$u$
	$$ F(u) = F\Big(\sum_{n\ge 1} \langle u,e_{\mu_n}\rangle e_{\mu_n}\Big)\;.$$
	We thus deduce that $F$ coincides $m$-a.e.~with a $\cF_\infty$-measurable map. Consequently $F=m[F \mid \cF_\infty]= 0$ $m$-a.e.
    \end{proof}

\subsection{Equicontinuity}\label{Subsec:EquiZRP}

Assume that for any $v,w \in \Omega_{N,\overline{n}}$ there exists a coupling of the processes $u_N^{v}$ and $u_N^w$ starting respectively from $v$ and $w$ such that:
\begin{eqnarray}
  \label{eq:contrazione}
\E[ \|u_N^v(t,\cdot) - u_N^w(t,\cdot)\|_{L^1} ] \le \|v-w\|_{L^1}\;.  
\end{eqnarray}
Then we deduce that for any $f\in \Lip(E_N)$
\begin{align*}
	\vert \cR^N_\lambda f(v) - \cR^N_\lambda f(w) \vert &\le \int_0^\infty e^{-\lambda t} \E[\vert f(u_N^v(t,\cdot)) - f(u_N^w(t,\cdot)) \vert] dt\\
	&\le  [f]_{\Lip(E_N)} \int_0^\infty e^{-\lambda t} \E[\| u_N^v(t,\cdot) - u_N^w(t,\cdot)\|_{L^1} ] dt\\
	&\le \lambda^{-1} [f]_{\Lip(E_N)} \|v-w\|_{L^1}\;,
\end{align*}
and Assumption \ref{Ass:Equi} is proved.

  \begin{remark}
    Note that, if the r.h.s. of  \eqref{eq:contrazione} were  replaced by $\|v-w\|_{L^1}$ times a function that  grows less than exponentially in $t$, Assumption
    \ref{Ass:Equi} would still follow.
  \end{remark}

We are thus left with constructing the above coupling. Actually we will construct a coupling of all processes $u_N^{v}$, $v\in \Omega_{N,\overline{n}}$ such that
\begin{enumerate}
	\item it preserves the ordering of the height functions: if $v \ge w$ then $u_N^v(t,\cdot) \ge u_N^w(t,\cdot)$ at all times $t\ge 0$,
	\item for any $v,w$, at all times $t\ge 0$ $\E[ \|u_N^v(t,\cdot) - u_N^w(t,\cdot)\|_{L^1} ] \le \|v-w\|_{L^1}$.
\end{enumerate} 
To do so, consider a collection of independent Poisson clocks $P_{i,i+1}^{k}$, $i\in\{1,\ldots,n-1\}$, $k\ge 1$ and $P_{i,i-1}^k$, $i\in\{2,\ldots,n\}$, $k\ge 1$ of rates $\frac12 (\tau(k) - \tau(k-1))$. With these Poisson clocks at hand, we can construct the processes $n_N^v$ in the following way. If {$P_{i,i\pm 1}^k$} rings at time $t$ then for all $v$ such that $n_N^{v}(t_-,i) \ge k$ we apply the following transition: $n_N^{v}(t,i) = n_N^{v}(t_-,i)-1$ and $n_N^{v}(t,i\pm1) = n_N^{u_0}(t_-,i\pm1)+1$.\\
It is then elementary to check that this coupling preserves the ordering of the height functions. Regarding the second property, we first consider the particular case where $v \ge w$. Then at all times $t\ge 0$
$$ \| u_N^w(t,\cdot) - u_N^v(t,\cdot)\|_{L^1} = \sum_{i=1}^N \frac1{N} \Big(u_N^v(t,i) - u_N^w(t,i)\Big)\;.$$
Consequently
$$ \partial_t \E[ \|u_N^v(t,\cdot) - u_N^w(t,\cdot)\|_{L^1} ] = \sum_{i=1}^{N-1} \frac1{N} \E[ \cL_N \Big(u_N^v(t,i)- u_N^w(t,i)\Big)] \;.$$
Since
$$ \E[\cL_N  u_N^v(t,i)] = \frac{N^{3/2}}2 \E[\tau(n_N^v(t,i+1)) - \tau(n_N^v(t,i))] \;,$$
and similarly with $w$, we obtain
\begin{multline} \partial_t \E[ \|u_N^v(t,\cdot) - u_N^w(t,\cdot)\|_{L^1} ]\\ = \frac{\sqrt N}2 \E[\tau(n_N^v(t,N))-\tau(n_N^w(t,N)) ] - \frac{\sqrt N}2 \E[\tau(n_N^v(t,1))-\tau(n_N^w(t,1))]\;.
\end{multline}
Since $u_N^v(t,\cdot) \ge u_N^w(t,\cdot)$ and since the total number
of particles is the same in both configurations, we deduce that
$$ n_N^v(t,N) \le n_N^w(t,N) \;,\quad n_N^v(t,1) \ge n_N^w(t,1)\;.$$
Since $k\mapsto \tau(k)$ is non-decreasing, we deduce that $\partial_t \E[ \|u_N^v(t,\cdot) - u_N^w(t,\cdot)\|_{L^1} ] \le 0$ so  that $\E[ \|u_N^v(t,\cdot) - u_N^w(t,\cdot)\|_{L^1} ] \le \|v-w\|_{L^1}$.\\
Let us now consider the general case where $v,w$ are not necessarily ordered. We introduce the configurations
$$ \max(i) := v(i) \vee w(i)\;,\quad \min(i) := v(i) \wedge w(i)\;.$$
Our coupling ensures that
$$ u_N^{\min}(t,\cdot) \le u_N^v(t,\cdot), u_N^w(t,\cdot) \le u_N^{\max}(t,\cdot)\;.$$
In addition
$$ \| v-w \|_{L^1} = \|\max - \min \|_{L^1}\;.$$
From the first part of the argument, we thus deduce
$$ \E[ \|u_N^v(t,\cdot) - u_N^w(t,\cdot)\|_{L^1} ] \le \E[ \|u_N^{\max}(t,\cdot) - u_N^{\min}(t,\cdot)\|_{L^1} ] \le \|\max - \min \|_{L^1} =  \| v-w \|_{L^1}\;.$$

\subsection{Discrete characterization and convergence}\label{Subsec:IbPFZRP}

The goal of this subsection is to check that Assumption \ref{Ass:IbPF} is fulfilled. Recall the parameters introduced above Lemma \ref{Lemma:Gaussian}. We set
\[
c := \frac{\rho}{2\alpha} \;,\quad \sigma := \sqrt{\overline{\tau}}\;,
\]
and we let $m = m_{c,\sigma}$ be the law on $E$ of a scaled Brownian bridge, as introduced in Subsection \ref{Subsec:SPDEZR}.

Recall that $\cC$ is the set of all maps $\psi_\varphi(u) = \exp( i \langle u,\varphi \rangle)$, $u\in E$, where $\varphi$ is a $C^2$ function from $[0,1]$ into $\R$ compactly supported in $(0,1)$. Recall the definition of the measure $\Sigma^{\psi_\varphi}$ from Subsection \ref{Subsec:SPDEZR}
\begin{equation*}
	\Sigma^{\psi_\varphi}(du) = \Big( ic \langle u,\varphi''\rangle -\frac{\sigma^2}{2}  \| \varphi\|_{L^2}^2 \Big) m(du)\;.
\end{equation*}
We set
$$ \psi_\varphi(u) \Sigma_N^{\psi_\varphi}(du) := \mathcal{L}_N \psi_\varphi(u)\cdot m_N(du)\;.$$

\begin{proposition}\label{Prop:DiscreteIbPF}
	The measure $\Sigma_N^{\psi_\varphi}$ converges weakly towards $\Sigma^{\psi_\varphi}$ in the set of finite complex measures on $E$.
\end{proposition}

\begin{proof}
	We need to show that for any bounded and continuous function $F:E\to\R$
	$$ \int_E F(u) \Sigma_N^{\psi_\varphi}(du) \to \int_E F(u) \Sigma^{\psi_\varphi}(du)\;,\quad N\to\infty\;.$$
	From now on, $F$ will implicitly be such a function. Moreover, $\varphi$ is a $C^2$ function compactly supported in $(0,1)$. In the computations below, some boundary terms will therefore vanish provided $N$ is large enough: we will always assume this is the case, and we will not display these boundary terms.\\
	We have
	\begin{align*}
		\mathcal{L}_N \psi_\varphi(u) &= \sum_{k = 1}^{N-1} N^2\tau(n_k) \psi_\varphi(u) \left( \frac{e^{\frac{i}{\sqrt{N}}\int_{\frac{k-1}{N}}^{\frac{k}{N}} \varphi(x)dx} + e^{-\frac{i}{\sqrt{N}}\int_{\frac{k}{N}}^{\frac{k+1}{N}} \varphi(x)dx}}{2} -1 \right)\\
		&= \sum_{k = 1}^{N-1} N^2\tau(n_k) \psi_\varphi(u) (I_1 + I_2 + I_3)\;,
	\end{align*}
	where
	$$ I_1 = - \frac{i}{2N^{3/2}} \Big(\varphi\Big(\frac{k}{N}\Big) - \varphi\Big(\frac{k-1}{N}\Big)\Big) \;,\quad I_2 = - \frac1{2N^3} \varphi^2\Big(\frac{k}{N}\Big) \;,\quad I_3 = O(N^{-7/2})\;.$$
	Let us control the term provided by $I_2$, that is
	\begin{align*}
		&-\int_E \frac1{2N} \sum_{k = 1}^{N-1} \tau(n_k) \psi_\varphi(u) \varphi^2(k/N) F(u) m_N(du)\\
		=& -\frac{\overline{\tau}}{2} \int_E \frac1{N} \sum_{k = 1}^{N-1} \psi_\varphi(u) \varphi^2(k/N) F(u) m_N(du)\\
		&-\frac1{2N} \int_E \sum_{k = 1}^{N-1} (\tau(n_k)-\overline{\tau}) \psi_\varphi(u) \varphi^2(k/N) F(u) m_N(du)\;.
	\end{align*}
	The first term on the r.h.s.~converges towards
	$$ - \frac{\overline{\tau}}{2} \int_E F(u)  \|\varphi\|^2 m(du)\;.$$
	The second term equals
	\begin{align*}
		-\frac1{2\sqrt{N}} \int_E \psi_\varphi(u) F(u) \Big( \sum_{k = 1}^{N-1} q_N(\frac{k}{N}) [ \varphi^2(\frac{k}{N}) -  \varphi^2(\frac{k+1}{N})] \Big)m_N(du)\;.
	\end{align*}
	The last bound stated in Lemma \ref{Lemma:Gaussian} ensures that this quantity converges to $0$ as $N\to\infty$.
	Regarding $I_1$, we first compute
	\begin{align*}
	&-\sum_{k = 1}^{N-1} N^2\tau(n_k) \psi_\varphi(u) \frac{i}{2N^{3/2}} \big(\varphi(k/N) - \varphi((k-1)/N)\big)\\
	= &-\frac{i}{2N}\sum_{k = 1}^{N-1} N^{1/2}(\tau(n_k)-\overline{\tau}) \psi_\varphi(u)  N \big(\varphi(k/N) - \varphi((k-1)/N)\big)\\
	=&\frac{i}{2N}\sum_{k = 1}^{N-1} q_N(k/N) N^2\Delta_N \varphi(k/N) \psi_\varphi(u)\;,
	\end{align*}
	where $\Delta_N \varphi(k/N) = \varphi((k+1)/N) - 2  \varphi(k/N) +  \varphi((k-1)/N)$. By Lemma \ref{Lemma:Gaussian}, this expression, multiplied by $F$ and integrated against $m_N$, converges to
	$$ \frac{i}{2} \E\big[ F(\beta) \langle \zeta, \varphi''\rangle \psi_\varphi(\beta) \big]\;.$$
	Since $\zeta- \frac{\rho}{\alpha} \beta$ is a centered Gaussian process independent of $\beta$, we deduce that this last expression coincides with
	$$ \frac{i}{2} \E\big[ F(\beta) \langle \frac{\rho}{\alpha}\beta , \varphi''\rangle \psi_\varphi(\beta) \big] = \frac{i}{2} \int F(u) \frac{\rho}{\alpha}\langle u, \varphi''\rangle \psi_\varphi(u) m(du)\;.$$
	Finally, the term $I_3$ has a negligible contribution. 
\end{proof}

\subsection{Tightness}\label{Subsec:TightZRP}

Let $H^{-\beta}$ be the Sobolev space of all distributions $f$ on $(0,1)$ such that
$$ \|f\|_{H^{-\beta}}^2 := \sum_{n\ge 1} n^{-2\beta} \langle f,e_n\rangle^2 < \infty\;,$$
where $(e_n)_{n\ge 1}$ is the usual eigenbasis of the Dirichlet Laplacian on $[0,1]$.

Note that $L^1$ is continuously embedded in $H^{-\beta}$ as soon as $\beta > 1/2$. The next proposition shows that Assumption \ref{Ass:Tight} is satisfied with $E^-=H^{-\beta}$ for any $\beta > 1$.

\begin{proposition}
	Take $\beta > 1$. The sequence $u_N$, $N\ge 1$ is tight in $\bbD([0,\infty),H^{-\beta})$ and all limiting points are in $C([0,\infty),H^{-\beta})$.
\end{proposition}
\begin{proof}
	To obtain tightness in $\bbD([0,\infty), H^{-\beta})$ with limiting points in $\bbC([0,\infty),H^{-\beta})$ it suffices to show that the sequence of initial laws $u_N(0,\cdot)$ is tight and that for all $T,\delta > 0$ we have
	\begin{equation}\label{Eq:TightnessZRP}
		\lim_{h\downarrow 0} \varlimsup_{N\rightarrow\infty} \frac{1}{h} \sup_{0 \leq s \leq T} \bbP\Big( \sup_{s\leq t \leq s+h} \| u_N(t) - u_N(s) \|_{H^{-\beta}} > \delta \Big) = 0\;,
	\end{equation}
	see for instance~\cite[Thm~13.2]{Billingsley}. The former is immediate since we start from the stationary measure $m_N$ that weakly converges to $m$. To check the latter condition, we apply the so-called Lyons-Zheng decomposition~\cite{LyonsZheng} that exploits the reversibility of the process w.r.t.~its invariant measure. Fix some arbitrary $T>0$. Let $\cF_t,t\geq 0$ be the filtration associated with $u_N(t),t\geq 0$, and let $\tilde{\cF}_t,t\in[0,T]$ be the filtration associated with the reversed process $u_N(T-t),t\in[0,T]$. For any $n\ge 1$, the process $\hat{u}_N(t,n):=\langle u_N(t), e_n\rangle$ satisfies
	\begin{equation}
		\hat{u}_N(t,n)-\hat{u}_N(s,n) = \int_s^t \cL_N \hat{u}(r,n) dr + M_N(t,n) - M_N(s,n)\;,
	\end{equation}
	where $M_N(t,n)$ is an $\cF_t$-martingale. Using the reversibility of the process, it is simple to get the identity
	\begin{equation}
		\hat{u}_N(T-(T-s),n)-\hat{u}_N(T-(T-t),n) = \int_{T-t}^{T-s} \cL_N \hat{u}_N(T-r,n) dr + \tilde{M}_N(T-s,n) - \tilde{M}_N(T-t,n)\;,
	\end{equation}
	where $\tilde{M}_N(t,n)$ is an $\tilde{\cF}_t$ martingale. Therefore, 
	\begin{equation}
		\hat{u}_N(t,n)-\hat{u}_N(s,n) =\frac12\Big( M_N(t,n) - M_N(s,n) + \tilde{M}_N(T-t,n) -\tilde{M}_N(T-s,n) \Big)\;.
	\end{equation}
	Using the Burkh\"older-Davis-Gundy inequality (see for instance~\cite{BDG}) to the martingales we deduce that for all $p\geq 2$
	\begin{align*}
		\bbE^N \Big[\sup_{t\in [s,s+h]}\big|M_N(t,n) - M_N(s,n)\big|^p\Big]^{\frac{1}{p}} &\lesssim \bbE^N \Big[\sup_{t\in [s,s+h]} \big\langle M_N(\cdot,n) - M_N(s,n)\big\rangle_t^{\frac{p}{2}}\Big]^{\frac{1}{p}}\\
		&\;+ \bbE^N\Big[\sup_{r\in(s,s+h]} \big|M_N(r,n)-M_N(r-,n)\big|^p\Big]^{\frac{1}{p}}\;,
	\end{align*}
	uniformly over all $N\geq 1$, all $s,h \ge 0$ and all $n\geq 1$. The jumps of $M_N$ are of size at most $N^{-3/2} \vert e_n' \vert=O(n N^{-3/2})$  so that the second term is easily bounded. Regarding the first term, observe that the predictable bracket is non-decreasing in time so that it suffices to take $t=s+h$. Now observe that there exists $C>0$ such that for all $N$ and $n$
	\begin{align*}
	\big\langle M_N(\cdot,n) - M_N(s,n)\big\rangle_{s+h} &= \int_s^{s+h} \frac1{N} \sum_{k=1}^N \frac12 N^2 \tau({n(r,k)} \Big( \Big(e_n(\frac{k}{N}) - e_n(\frac{k-1}{N})\Big)^2\\ + \Big(e_n(\frac{k+1}{N}) - e_n(\frac{k}{N})\Big)^2 \Big) dr
	&\le C n^2\int_s^{s+h}  \frac1{N} \sum_{k=1}^N\tau({n(r,k)}) dr.
	\end{align*}
	By Jensen's inequality
	$$ \Big\vert \int_s^{s+h}  \frac1{N} \sum_{k=1}^N\tau({n(r,k)}) dr \Big\vert^{\frac{p}{2}} \le h^{\frac{p}{2}-1} \int_s^{s+h}  \frac1{N} \sum_{k=1}^N\tau({n(r,k)})^{p/2} dr\;.$$
	Noting that the law of $\{n(i), i\le N\}$ is exchangeable under $m_N$ we get
	$$ \bbE^N \Big[\big\langle M_N(\cdot,n) - M_N(s,n)\big\rangle_{s+h}^{\frac{p}{2}}\Big]^{\frac{1}{p}} \lesssim h^{\frac12} n\,m_N[(\tau({n(1)}))^{p/2}]^{1/p}\;.$$
	The moment on the r.h.s.~is controlled by the last estimate of Lemma \ref{Lemma:Gaussian}. Consequently,
	\[\bbE^N \Big[\sup_{t\in [s,s+h]}\big|M_N(t,n) - M_N(s,n)\big|^p\Big]^{\frac{1}{p}} \lesssim n(h^{\frac{1}{2}} + N^{-\frac{3}{2}})\;.\]
	We deduce that
	\begin{equation*}
		\bbE^N \Big[\sup_{s\leq t \leq s+h} \big|\hat{u}_N(t,n)-\hat{u}_N(s,n)\big|^p\Big]^{\frac{1}{p}} \lesssim n(h^{\frac{1}{2}} + N^{-\frac{3}{2}})\;,
              \end{equation*}
	uniformly over all $N\geq 1$, $s\in [0,T-h]$ and $n\geq 1$. Applying the triangle inequality w.r.t.~the $L^{p/2}(\Omega)$-norm at the second line we get
	\begin{align*}
		\bbE^N \Big[\sup_{s\leq t \leq s+h} \big\|u_N(t)-u_N(s)\big\|^p_{H^{-\beta}}\Big]^{\frac{1}{p}} &\leq \bbE^N \bigg[\Big(\sum_{n\geq 1} n^{-2\beta}\sup_{s\leq t \leq s+h} \big|\hat{u}_N(t,n)-\hat{u}_N(s,n)\big|^{2}\Big)^{p/2}\bigg]^{\frac{1}{p}}\\
		&\leq \bigg(\sum_{n\geq 1} n^{-2\beta} \bbE^N \Big[\sup_{s\leq t \leq s+h} \big|\hat{u}_N(t,n)-\hat{u}_N(s,n)\big|^{p}\Big]^{\frac{2}{p}}\bigg)^{\frac{1}{2}}\\
		&\lesssim h^{\frac{1}{2}} + N^{-\frac{3}{2}}\;,
	\end{align*}
    if $2\beta-1>1$ so that \eqref{Eq:TightnessZRP} follows provided $p\ge 3$.
\end{proof}

We have therefore completed the proof of the following result.

\begin{theorem}\label{Th:ZRP}
	The process $(u_N(t), t\ge 0)$, starting from its stationary measure $m_N$, converges in law to the process $(u(t),t\ge 0)$, starting from $m_{c,\sigma}$, in the space $\bbD([0,\infty),H^{-\beta}([0,1]))$ for any given $\beta > 1$.
\end{theorem}

\section{Pair of reflected interfaces}\label{Sec:Reflected}

The dynamics studied in this section is a special case of a dynamics of $k\ge 2$ reflected interfaces,
introduced previously in \cite{luby2001markov,wilson2004mixing} in the context of dynamics of planar structures and tilings. See also Remark \ref{rem:k>2} below.

We consider the set $\Omega_N$ of all pairs $\ttu=(\ttv,\ttw)$ of lattice paths from $\{0,\ldots,2N\}$ to $\bbZ$ that make $\pm 1$ steps, are pinned at height $0$ at both ends and are ordered, namely
\begin{align*}
	\ttv(k+1)-\ttv(k) \in \{-1,1\}\;,\quad \ttw(k+1)-\ttw(k) \in \{-1,1\}\;,\\
	\ttv(0)=\ttw(0)=\ttv(2N)=\ttw(2N) = 0\;,\\
	\ttv \ge \ttw\;.
\end{align*}
We let $m_N$ be the uniform measure on $\Omega_N$. Let $\Delta$ be the discrete one-dimensional Laplacian, that is
\[
	\Delta f(k) = f(k+1) - 2 f(k) + f(k-1)\;.
\]
Given a configuration $\mathtt u=(\ttv,\ttw) \in \Omega_N$, we say that $k\in\{1,\ldots,2N-1\}$ is a \emph{contact point} if
$$ \ttv(\ell) = \ttw(\ell)\;,\quad \forall \ell \in \{k-1,k,k+1\}\;.$$
We also say that $\ttv$ forms an upward corner at $k$ if $\Delta \ttv = -2$ and a downward corner at $k$ if $\Delta \ttv = 2$. The same definition holds for $\ttw$. On $\Omega_N$, we introduce a partial order by saying that $\ttu=(\ttv,\ttw)\ge \ttu'=(\ttv',\ttw')$
iff $\ttv(k)\ge \ttv'(k)$ and $\ttw(k)\ge \ttw'(k)$ for every $k$. Given a lattice path $\ttv$, if $\ttv$ forms an (upward or downward) corner at $k$, we let $\ttv^k$ denote the configuration where $\ttv(k)$ is flipped to $\ttv(k)+\Delta \ttv(k)$.\\
Note that, if $k$ is a contact point and $\ttv,\ttw$ form an upward (resp.~downward) corner there, then the configuration $(\ttv^k,\ttw)$ does not belong to $\Omega_N$, but $(\ttv,\ttw^k)$ does (resp.~$(\ttv,\ttw^k)$ does not belong to $\Omega_N$, but $(\ttv^k,\ttw)$ does).\\

We consider the following corner flip dynamics on $\Omega_N$:\begin{enumerate}
	\item for every $k\in \{1,\ldots,2N-1\}$ which is not a contact point, if $\ttv$ (resp.~$\ttw$) forms a corner at $k$, then $\ttv(k)$ (resp.~$\ttw(k)$) is flipped at rate $(2N)^2/2$, so that the resulting configuration is  $(\ttv^k,\ttw)$ (resp.~$(\ttv,\ttw^k)$);
	\item for every contact point $k$, if $\ttv$ and $\ttw$ form an upward corner (resp.~downward corner) then $\ttw(k)$ (resp.~$\ttv(k)$) flips at rate $(2N)^2/2$, so that the resulting configuration is  $(\ttv,\ttw^k)$ (resp.~$(\ttv^k,\ttw)$;
	\item for every contact point $k$, if $\ttv$ and $\ttw$ form a corner then they flip simultaneously at rate $(2N)^2/4$, so that the resulting configuration is $(\ttv^k,\ttw^k)$.
\end{enumerate}
We refer to Figure \ref{Fig:Contact} for an illustration. Let us emphasize that all these jumps are driven by independent Poisson clocks. It is straightforward to check that this dynamics is reversible w.r.t.~the uniform measure $m_N$ on $\Omega_N$.

\begin{remark}
  \label{rem:letus}
	Let us point out that the jump rates at contact points are chosen in such a way that the generator satisfies
	$$  \cL_N[\ttv(k)+\ttw(k)] = \frac{(2N)^2}{2} \Delta (\ttv(k) + \ttw(k))\;.$$
	This identity would still hold if we imposed:
	\begin{enumerate}\setcounter{enumi}{1}
		\item for every contact point $k$, if $\ttv$ and $\ttw$ form an upward corner (resp.~downward corner) then $\ttw(k)$ (resp.~$\ttv(k)$) flips at rate $\alpha$,
		\item for every contact point $k$, if $\ttv$ and $\ttw$ form a corner then they flip simultaneously at rate $\beta$
	\end{enumerate}
with $\alpha,\beta \ge 0$ satisfying $\frac{\alpha}{2} + \beta  = (2N)^2/2$. In this more general setting, the invariant measure is no longer uniform nor reversible (except if $\alpha = (2N)^2/2$ and $\beta=(2N)^2/4$).
\end{remark}

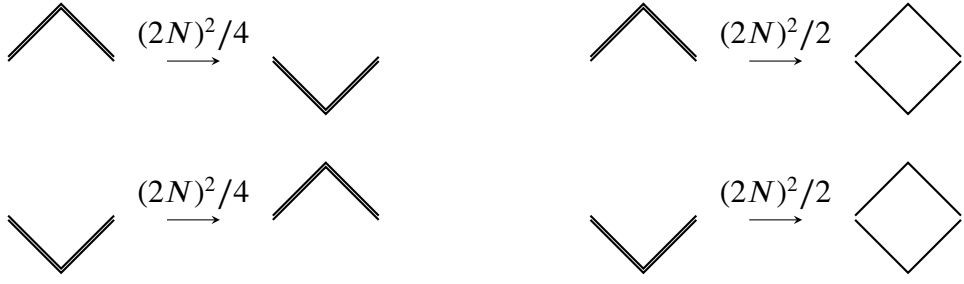
\begin{figure}\centering
	\begin{tikzpicture}[xscale=0.7, yscale=0.7, >=stealth]\label{Fig2}
		
		\draw[-,thick,color=black] (2,3.08) -- (3,4.08) -- (4,3.08);
		\draw[-,thick,color=black] (2,3) -- (3,4) -- (4,3);
		
		\draw[->,color=black] (5,3) -- (5.5,3)node[above]{$(2N)^2/4$} -- (6,3);
		
		\draw[-,thick,color=black] (7,3.08) -- (8,2.08) -- (9,3.08);
		\draw[-,thick,color=black] (7,3) -- (8,2) -- (9,3);
		
		
		\draw[-,thick,color=black] (13,3.08) -- (14,4.08) -- (15,3.08);
		\draw[-,thick,color=black] (13,3) -- (14,4) -- (15,3);
		
		\draw[->,color=black] (16,3) -- (16.5,3)node[above]{$(2N)^2/2$} -- (17,3);
		
		\draw[-,thick,color=black] (18,3.08) -- (19,4.08) -- (20,3.08);
		\draw[-,thick,color=black] (18,3) -- (19,2) -- (20,3);
		
		
		\draw[-,thick,color=black] (2,0.08) -- (3,-0.92) -- (4,0.08);
		\draw[-,thick,color=black] (2,0) -- (3,-1) -- (4,0);
		
		\draw[->,color=black] (5,0) -- (5.5,0)node[above]{$(2N)^2/4$} -- (6,0);
		
		\draw[-,thick,color=black] (7,0.08) -- (8,1.08) -- (9,0.08);
		\draw[-,thick,color=black] (7,0) -- (8,1) -- (9,0);
		
		
		\draw[-,thick,color=black] (13,0.08) -- (14,-0.92) -- (15,0.08);
		\draw[-,thick,color=black] (13,0) -- (14,-1) -- (15,0);
		
		\draw[->,color=black] (16,0) -- (16.5,0)node[above]{$(2N)^2/2$} -- (17,0);
		
		\draw[-,thick,color=black] (18,0.08) -- (19,1.08) -- (20,0.08);
		\draw[-,thick,color=black] (18,0) -- (19,-1) -- (20,0);

	\end{tikzpicture}
	\caption{Dynamics at contact points}\label{Fig:Contact}
\end{figure}

In order to take the continuum limit $N\to\infty$, we need to suitably rescale the paths. Given $\ttu\in\Omega_N$, we define
  $u=(v,w)$ where $v:[0,1]\mapsto \mathbb R$ is the continuous function defined via
  \begin{eqnarray}
    \label{eq:tildeu}
    v(k/(2N)):=\frac1{\sqrt {2N}}\ttv(k)\;,\quad k\in \{0,1,\dots,2N\}
  \end{eqnarray}
  and by linear interpolation in between. For lightness of notation, we do not give an index $N$ to the rescaled functions, in contrast with what we did for the zero-range process. The function $w$ is defined analogously. Then, $u$ is an element of the space $E$ of all pairs $u=(v,w)$ of continuous functions from $[0,1]$ into $\mathbb R$ such that $v\ge w$. On $E$, we consider the distance
$\|u-u'\|_E := \|v-v'\|_{L^1} + \|w-w'\|_{L^1}$. Note that $E$ is a separable metric space, but it is not complete, see the next subsection for a discussion on that point. For fixed $N$, $u$ belongs to the finite set $E_N\subset E$ (obviously defined as the image of $\Omega_N$ via the rescaling). Also, $u^{u_0}(t)$ denotes the rescaled configuration at time $t$, with rescaled initial configuration $u_0$.

\subsection{The SPDE}\label{Sec:SPDEPair}

The limiting dynamics is better described in terms of the sum and difference of the interfaces. Let us introduce the associated stochastic PDEs. Given a space-time white noise $\xi^S$ on $\R_+\times[0,1]$ we consider the solution $S$ of the following stochastic PDE
\begin{align*}
	\begin{cases}
		\partial_t S(t,x) = \frac12 \partial_x^2 S(t,x) + \xi^S\;,\quad x\in [0,1]\;,\quad t\ge 0\;,\\
		S(t,0)=S(t,1)=0\;,
	\end{cases}
\end{align*}
where the initial condition $S(0,\cdot)$ is any continuous function on $[0,1]$. This equation should be understood in the weak sense: for any $C^2$ function $h:[0,1]\to\R$ that vanishes at $0$ and $1$, we have
\begin{equation}\label{Eq:Sweak}
	\langle S(t,\cdot),h\rangle = \langle S(0,\cdot),h\rangle + \frac12\int_0^t \langle S(t,\cdot),h''\rangle +\int_0^t h(x) \xi^S(ds,dx)\;.
\end{equation}
It is well-known that there exists a unique solution $S$ which is continuous in time with values in $C([0,1])$ (endowed with the sup-norm, i.e. the solution is continuous in both variables on $[0,+\infty)\times [0,1]$). Moreover the solution $(S(t,\cdot))_{t\geq 0}$ is a Markov process which admits a unique invariant (reversible) measure, given by the law $\mBB$ of the standard Brownian bridge on $[0,1]$.\\

Given a space-time white noise $\xi^D$, independent of $\xi^S$, we consider the solution $(D,\eta)$ of the following stochastic PDE
\begin{align}	\begin{cases}
		\partial_t D(t,x) = \frac12 \partial_x^2 D(t,x) + \xi^D + \eta\;,\quad x\in [0,1]\;,\quad t\ge 0\;,\\
		D(t,x)\ge 0\;,\quad \;\quad \int_{t\in\R_+} \int_{x\in (0,1)} D(t,x) \eta(dt,dx) = 0\;,\\
		D(t,0)=D(t,1)=0\;,
              \end{cases}
  \label{eq:NP}
\end{align}
where the initial condition $D(0,\cdot)$ is any given non-negative continuous function. Here $\eta$ is a random non-negative Radon measure on $\R_+\times (0,1)$ and the equation is understood similarly as in \eqref{Eq:Sweak} except that one considers the additional term $\int_0^t h(x) \eta(ds,dx)$. The random variable $\eta$ is a reflection measure: it is supported by the space time points $(t,x)$ where the interface $D$ hits $0$, and at these points, it pushes the interface upwards to keep it non-negative.\\
Nualart and Pardoux~\cite{NualartPardoux92} showed that almost surely, there exists a unique pair $(D,\eta)$ satisfying the system above (note that the construction therein is purely deterministic). The process $(D(t,\cdot))_{t\geq 0}$ is Markov with values in $C_+([0,1]):=\{f\in C([0,1]): f\geq 0\}$ and admits a unique invariant (reversible) measure, given by the law $\mBE$ of the normalised Brownian excursion, see~\cite{Zambotti2001,MR3616274}.\\
Since $(S,D)$ is an independent pair of Markov processes, then the pair $(S,D)$ is itself a Markov process on $E:=C([0,1])\times C_+([0,1])$ that admits an invariant and reversible invariant measure $\mBB \otimes \mBE$.

In this situation, $E$ has a natural metric given by the sup-norm, which makes it a Polish space and $(S,D)$ an a.s.~continuous process. However we have allowed ourselves
in the black box result of Section \ref{sec:ageneral} to endow our state space $E$ with a distance $d$ which makes it separable but not necessarily complete. Here we exploit
this freedom and we choose $d$ as the $L^1([0,1],dx)$ distance: for $u=(u^1,u^2), v=(v^1,v^2)\in E=C([0,1])\times C_+([0,1])$
\[
\|u-v\|_E:=\sum_{i=1}^2\|u^i-v^i\|_{L^1([0,1],dx)}.
\] 
The reason for this choice is that the contraction property \eqref{Eq:DecayPair} for the discrete processes can be
proved in this weaker distance.\\

We then define $u=(v,w)$ through
$$ v := \frac{S+D}{\sqrt 2} \;,\quad w = \frac{S-D}{\sqrt 2}\;.$$
It solves
\begin{align}
	\begin{cases}
		\partial_t v = \frac12\partial^2_x v + \xi^v + \frac1{\sqrt 2}\eta\;,\quad x\in [0,1]\;,\quad t\ge 0\;,\\
		\partial_t w = \frac12\partial^2_x w + \xi^w - \frac1{\sqrt 2}\eta\;,\quad x\in [0,1]\;,\quad t\ge 0\;,\\
		v\ge w\;,\quad \eta(dt,dx)\ge 0\;,\quad \int_{(0,1)\times \R_+} (v-w)(t,x) \eta(dt,dx) = 0\;,
              \end{cases}
  \label{eq:theSPDEs}
\end{align}
where $\xi^v := \frac{\xi^S+\xi^D}{\sqrt 2}$ and $\xi^w:=\frac{\xi^S - \xi^D}{\sqrt 2}$ are still independent space-time white noises.

We let $m$ be the pushforward of $\mBB\otimes \mBE$ through the map $(S,D) \mapsto (\frac{S+D}{\sqrt 2} , \frac{ S-D }{\sqrt 2})$. The process $u=(v,w)$ is a Markov process taking values in $E$ with invariant (reversible) measure $m$. Our next proposition provides a characterization of its resolvents and thus establishes Assumption \ref{Ass:Charac} in the present model. This proposition takes advantage of the independence of $S$ and $D$, and therefore mixes the two sets of coordinates ($(S,D)$ and $(v,w)$).\\

For any pair of continuous functions $(\varphi_v,\varphi_w)$ from $[0,1]$ into $\R$, we define for any $u\in E$
$$ \psi_\varphi(u):= e^{i \langle v,\varphi_v\rangle+i\langle w,\varphi_w\rangle}\;.$$
We call $(\cR_\lambda)_{\lambda>0}$ the
resolvent family of the Markov process $(S,D)$.
\begin{proposition}
  \label{prop:Sigma2int}
	Let $f$ be a bounded Lipschitz function from $E$ to $\R$. For any $\lambda > 0$, $\cR_\lambda f$ is the unique bounded Lipschitz-continuous function $F:E\to\R$ such that for any pair of $C^2$ functions $(\varphi_v,\varphi_w)$ from $[0,1]$ into $\R$ with compact supports in $(0,1)$, the following holds:
	\begin{equation}\label{Eq:CharactResolPair}
		\int F(u) \psi_\varphi(u) (\lambda m(du) - \Sigma^{\psi_\varphi}(du)) = \int f(u) \psi_\varphi(u) m(du)\;,
	\end{equation}
	where, writing $S=\frac{v+w}{\sqrt{2}}$ and $D= \frac{v-w}{\sqrt{2}}$
	\begin{equation}\label{Eq:SigmaPair}\begin{split}
		\Sigma^{\psi_\varphi}(du) =& \frac12 \Big( i \langle v,\varphi''_v\rangle  + i \langle w,\varphi''_w\rangle - \| \varphi_v\|_{L^2}^2 -  \| \varphi_w\|_{L^2}^2\Big) m(du)\\
		&+\frac{i}{2} \int_0^1 dr \frac{1}{\sqrt{2\pi r^3(1-r)^3}} \frac{(\varphi_v-\varphi_w)}{\sqrt{2}}(r) \mBE(dD \mid D(r) = 0)\mBB(dS)\;.
	\end{split}
	\end{equation}
      \end{proposition}
      In this case, we let $\mathcal C$ denote the set of all functions $\psi_\varphi(\cdot)$ with $\varphi$ as in the statement of the proposition. Let us point out that the probability measure $\mBE(dD \mid D(r) = 0)$ can be defined by passing to the limit on a regular conditioning of the form $\{D(r) < \varepsilon\}$. This probability measure then concides with the law of the concatenation of two Brownian excursions rescaled on $(0,r)$ and $(r,1)$ respectively: we refer to Step 3 in the proof of Proposition \ref{Prop:DiscreteIbPFPair} for more details in this direction.
      
\begin{proof}
By \cite[chapter V]{BH91}, the Markov process $(S,D)$ is associated with the tensor product $\cE$ of the Dirichlet forms $\cE^S,\cE^D$ of $S$ and $D$ respectively. Moreover $\cR_\lambda f$ is the unique function in $\cD(\cE)$ such that \eqref{eq:page16} holds for all $g\in\cD(\cE)$. Then if we choose $g=\psi_\varphi$, we obtain by integration by parts that $\cR_\lambda f$ does satisfy \eqref{Eq:CharactResolPair}, using the explicit form of the integration by parts for $\cE^D$ given in \cite{Zambotti2002}. On the other hand,
if $F:E\to\R$ is bounded Lipschitz-continuous then $F\in\cD(\cE)$ and if $F$ also satisfies \eqref{Eq:CharactResolPair} then by integration
by parts $F$ satisfies \eqref{eq:page16} for $g=\psi_\varphi$. Since exponential functions are a core in the domains $\cD(\cE^S)$ and
$\cD(\cE^D)$, then $F$ 
satisfies \eqref{eq:page16} for all $g\in\cD(\cE)$.
%
%
\end{proof}
      
\begin{lemma}\label{lemma:CVInvMeas}
	As $N\to\infty$, $m_N$ converges weakly to the measure $m$.
\end{lemma}
\begin{proof}
	It is known, see for instance~\cite[Theorem 2.7]{Gillet03}, that $m_N$ converges weakly to a probability measure on $C([0,1],\R^2)$ (and a fortiori on $E$) under which $(v,w)$ solves
	\begin{align*}
		dv(t) &= - \frac{v(t)}{1-t} dt + \frac{dt}{v(t)-w(t)} + dB^v(t)\;,\quad v(0) = 0\;,\\
		dw(t) &= - \frac{w(t)}{1-t} dt + \frac{dt}{w(t)-v(t)} + dB^w(t)\;,\quad w(0)=0\;,
	\end{align*}
	where $B^v$ and $B^w$ are independent standard Brownian motions. It is then straightforward to check that $S:=(v+w)/\sqrt{2}$ and $D:=(v-w)/\sqrt{2}$ solve
		\begin{align*}
		dS(t) &= - \frac{S(t)}{1-t} dt + dB^S(t)\;,\quad S(0) = 0\;,\\
		dD(t) &= - \frac{D(t)}{1-t} dt + \frac{dt}{D(t)} + dB^D(t)\;,\quad D(0)=0\;,
	\end{align*}
	where $B^S$ and $B^D$ are independent standard Brownian motions. As a consequence $S$ is a Brownian bridge and $D$ an independent normalised Brownian excursion.
\end{proof}
This lemma ensures that Assumption \ref{Ass:State} is satisfied.

\subsection{Equicontinuity}

We follow the same strategy as in Subsection \ref{Subsec:EquiZRP} in order to establish Assumption \ref{Ass:Equi}. It suffices to show that for any $u_0,u'_0 \in E_N$, there exists a coupling under which
\begin{align}
	\E\big[\|u^{u_0}(t) - u^{u_0'}(t) \|_{E}\big] &\le \|u_0-u_0'\|_E\;.\label{Eq:DecayPair}
\end{align}
From now on, we write $u$ and $u'$ for $u^{u_0}$ and $u^{u_0'}$.

There exists a global monotone coupling of the processes starting from
all initial configurations $\mathtt u_0\in\Omega_N$ such that if
$\mathtt u_0 \ge \mathtt u_0'$ then almost surely $\ttu^{u_0}(t) \ge \ttu^{u_0'}(t)$ for
all $t\ge 0$. Indeed, the dynamics we are considering is a special case of a dynamics of $k$ reflected interfaces (see Remark \ref{rem:k>2} below), for which the existence of a global monotone coupling is known \cite{wilson2004mixing}.\\
With this coupling at hand, let us consider first the case
$\mathtt u_0 \ge \mathtt u_0'$. Then at all times $t\ge 0$ 
$$(2 N)^{3/2}\| u(t) - u'(t)\|_E = \sum_{k=1}^{2N-1} \ttv(t,k)-\ttv'(t,k) + \ttw(t,k)-\ttw'(t,k)\;.$$
Consequently,
$$ (2 N)^{3/2}\partial_t \E\big[\| u(t) - u'(t)\|_E\big] = \sum_{k=1}^{2N-1} \E\big[ \cL_N(\ttv(t,k)-\ttv'(t,k) + \ttw(t,k)-\ttw'(t,k)) \big]\;.$$
Thanks to Remark \ref{rem:letus},
$$\E\big[\cL_N( \ttv(t,k)+ \ttw(t,k)) \big] =\frac{(2N)^2}{2}\E\big[ \Delta \ttv(t,k) + \Delta \ttw(t,k) \big]\;,$$
and similarly for $\ttv' + \ttw'$. Since the interfaces are pinned at height $0$ at $k=0$ and $k=2N$, an integration by parts gives
\begin{align*}
&(2 N)^{3/2} \partial_t \E\big[\| u_N(t) - u_N'(t)\|_E\big]\\
&= -\frac{(2N)^2}{2} \sum_{k\in \{1,2N-1\}} \E\big[\ttv(t,k)-\ttv'(t,k)+\ttw(t,k)-\ttw'(t,k)\big]\;,
\end{align*}
which is non-positive, and therefore \eqref{Eq:DecayPair} is proved.\\
To treat the general case where $\tu_0$ and $\tu'_0$ are not ordered, we introduce
$$ \tv^{\max}(k) = \max(\tv_0(k),\tv_0'(k))\;,\quad \tv^{\min}(k) = \min(\tv_0(k),\tv_0'(k))\;,$$
together with
$$ \Tw^{\max}(k) = \max(\Tw_0(k),\Tw_0'(k))\;,\quad \Tw^{\min}(k) = \min(\Tw_0(k),\Tw_0'(k))\;.$$
Our coupling implies that almost surely at all times $t\ge 0$
$$ \ttv^{\max}(k,t) \ge \ttv(k,t),\ttv'(k,t) \ge \ttv^{\min}(k,t)\;,$$
and
$$ \ttw^{\max}(k,t) \ge \ttw(k,t),\ttw'(k,t) \ge \ttw^{\min}(k,t)\;,$$
with the notation that $(\ttv^{\max}(t),\ttw^{\max}(t))$ is the configuration at time $t$ started from $(\ttv^{\max},\ttw^{\max})$, and analogously for the evolution starting from the minimal configuration.
Consequently,
$$\E[ \|u(t) - u'(t) \|_{E} ]\le \E[\|u_N^{\max}(t) - u_N^{\min}(t) \|_{E}] \le  \|u^{\max} - u^{\min} \|_{E}\;.$$
To conclude, one observes that
\begin{align*}
	\|u^{\max} - u^{\min} \|_{E} &= \|v^{\max} - v^{\min} \|_{L^1} + \|w^{\max} - w^{\min} \|_{L^1}\\
&= \|v_0 - v'_0\|_{L^1} +  \|w_0 - w'_0\|_{L^1}\\
&= \|u_0 - u_0'\|_E\;.
\end{align*}

\subsection{Discrete characterization and convergence}

In this subsection, we check that Assumption \ref{Ass:IbPF} is satisfied. Given two functions $\varphi_v,\varphi_w$ from $[0,1]$ into $\R$, recall that
$$ \psi_\varphi(u)= e^{i \langle v,\varphi_v\rangle+i\langle w,\varphi_w\rangle}\;.$$
We define the finite complex measure $ \Sigma^{\psi_\varphi}_N$ on $E$ by setting
\begin{equation}
	\Sigma^{\psi_\varphi}_N(du) =	\mathcal{L}_N \psi_\varphi(u)m_N(du)\;.
\end{equation}

\begin{proposition}\label{Prop:DiscreteIbPFPair}
	Let $\varphi_v,\varphi_w$ be two $C^2$ functions from $[0,1]$ into $\R$ with compact supports in $(0,1)$. As $N\to\infty$, $ \Sigma^{\psi_\varphi}_N$ converges weakly to the finite complex measure $\Sigma^{\psi_\varphi}$ which was defined in \eqref{Eq:SigmaPair}.
\end{proposition}
\begin{proof}
We write $v^k,w^k$, $k=1,\dots,2N-1$, for the configurations $\tv^k,\Tw^k$ after rescaling. Since $\tv,\Tw$ are obtained by affine interpolation in between their values at lattice sites, one can check that
$$ \langle v,\varphi_v\rangle = \frac1{(2N)^{3/2}} \sum_{k=1}^{2N-1} \tv(k) \Phi_v(k/(2N))\;,$$
with
$$ \Phi_v(k/(2N)) = 2N \int_{(k-1)/(2N)}^{k/(2N)} \varphi_v(x) (2Nx-(k-1)) dx + 2N \int_{k/(2N)}^{(k+1)/(2N)} \varphi_v(x) (k+1 - 2Nx) dx\;,$$
and similarly for $\langle w,\varphi_w\rangle$ with a function $\Phi_w$ defined analogously.

A direct computation shows that the generator $\mathcal{L}_N$ can be written as a sum of $ \mathcal{L}_N^A$ and $ \mathcal{L}_N^R$ where the former deals with non-reflected terms in the generator, that is
	\begin{align*}		
		\mathcal{L}_N^A \psi_\varphi(u) &= \sum_{k = 1}^{2N-1} {(2N)^2} \psi_\varphi(u)A_k(u)\;,
	\end{align*}
	with
	$$ A_k(u) = \frac12\Big(\frac{\psi_\varphi(v^k,w)}{\psi_\varphi(v,w)} - 1 \Big) \mathbf{1}_{\Delta \tv(k) \ne 0} + \frac12 \Big(\frac{\psi_\varphi(v,w^k)}{\psi_\varphi(v,w)} - 1 \Big) \mathbf{1}_{\Delta \Tw(k) \ne 0}\;,$$
	while $ \mathcal{L}_N^R$ is the part of the generator coming from the reflection, that is
	\begin{align*}		
		\mathcal{L}_N^R \psi_\varphi(u) &= \sum_{k = 1}^{2N-1} {(2N)^2}\psi_\varphi(u)R_k(u)\;,
	\end{align*}
	with
	\begin{align*}
		R_k(u) = &-\frac12\Big(\frac{\psi_\varphi(v^k,w)}{\psi_\varphi(v,w)} - 1 \Big) \mathbf{1}_{\Wedge(k)}-\frac12\Big(\frac{\psi_\varphi(v,w^k)}{\psi_\varphi(v,w)} - 1 \Big) \mathbf{1}_{\Vee(k)} \\
		& + \frac14\Big(\frac{\psi_\varphi(v^k,w^k)}{\psi_\varphi(v,w)} - 1 \Big)( \mathbf{1}_{\Wedge(k)}+\mathbf{1}_{\Vee(k)})\;.
	\end{align*}
    Here $\Wedge(k)$ (resp. $\Vee(k)$) is the event that $k$ is  a contact point and the interfaces $\tv, \Tw$ have an upward (resp. downward) corner there. We control these terms separately. From now on $O(N^{-a})$ will denote a quantity that is bounded by $CN^{-a}$ where $C>0$ is independent of $N$ and $k$.\\
    
    \noindent\emph{Step 1: non-reflected terms.} One can check that
              \begin{align*}
                A_k(u)                &=\frac i{2(2N)^{3/2}}[\Phi_v(k/(2N))\Delta \tv(k)+\Phi_w(k/(2N))\Delta \Tw(k)]\\&-\frac{1}{(2N)^3}\Big[\Big(\Phi_v\big(\frac k{2N}\big)\Big)^2 \mathbf{1}_{\Delta \tv(k) \ne 0}
+\Big(\Phi_w\big(\frac k{2N}\big)\Big)^2 \mathbf{1}_{\Delta \Tw(k) \ne 0}\Big] +O(N^{-9/2})\;.
              \end{align*}
    Via a summation by parts and recalling that $\varphi_v,\varphi_w$ are $C^2$ functions that have compact support in $(0,1)$, 
    we thus deduce that, up to negligible terms, $\mathcal{L}_N^A \psi_\varphi(u)$ coincides with
    \begin{equation}
    	\frac12 \psi_\varphi(u)\Big( i \langle v,\varphi_v''\rangle+ i \langle w,\varphi_w''\rangle -\|\varphi_v\|^2 -\|\varphi_w\|^2\Big) + \cS_N(u)\;,
    \end{equation}
    where
    \begin{multline}
    	\cS_N(u) := \psi_\varphi(u)\frac1{N}\sum_{k=1}^{2N-1}\Big(\Phi_v\Big(\frac k{2N} \Big)^2\Big(\frac12-\mathbf{1}_{\Delta\tv(k)\ne0}\Big)+
    	\Phi_w\Big(\frac k{2N} \Big)^2\Big(\frac12-\mathbf{1}_{\Delta\Tw(k)\ne0}\Big)\Big)\;.
    \end{multline}
    Let us show that $\cS_N(u) m_N(du)$ converges weakly to $0$ as $N\to\infty$. Set
    \[
    a_k:=\Phi_v^2(k/(2N))\;.
    \]
    By symmetry, and since $\psi_\varphi$ is a bounded function, it is enough to prove that
    \begin{equation}\label{Eq:ToproveSN}
    	\frac1N m_N\Big[|\sum_{k=1}^{2N-1}a_k \Big(\frac12-\mathbf{1}_{\Delta\tv(k)\ne0}\Big)|\Big]
    \end{equation}
    tends to zero as $N\to\infty$. Let $\mathbb P^{\otimes 2}$ denote the law of two iid simple random walks $\tv,\Tw$ started at zero, and set
    $$D_N := \{(\tv,\Tw):  \forall k\;,\tv(k) \ge \Tw(k) \quad \mbox{ and }\tv(2N)=\Tw(2N)\}\;.$$
    We then have $m_N(\cdot)=\mathbb P^{\otimes 2}(\cdot|D_N)$.\\
    Using Jensen's inequality for $p>1$ and the definition of $m_N$, \eqref{Eq:ToproveSN} is smaller than
    \begin{multline}
    	\frac1{\sqrt N} \left(m_N\Big[\Big|\frac1{\sqrt N}\sum_{k=1}^{2N-1}a_k \Big(\frac12-\mathbf{1}_{\Delta\tv(k)\ne0}\Big)\Big|^p\Big]
    	\right)^{1/p}\\\le \frac1{\sqrt N(\mathbb P^{\otimes 2}(D_N))^{1/p}}  \left(\mathbb E^{\otimes 2}\Big[\Big|\frac1{\sqrt N}\sum_{k=1}^{2N-1}a_k \Big(\frac12-\mathbf{1}_{\Delta\tv(k)\ne0}\Big)\Big|^p\Big]
    	\right)^{1/p}\;.
    \end{multline}
    Since $\mathbb P^{\otimes 2}(D_N)$ is of order $1/N^2$, the constants $a_k$ are bounded uniformly in $N$ and the random variables $\frac12-\mathbf{1}_{\Delta\tv(k)\ne0}$ are i.i.d., centered and of finite variance under $\mathbb P^{\otimes 2}$, this expression is upper bounded by a $p$-dependent constant times 
    $N^{-1/2+2/p} $, that tends to zero as soon as $p>4$.\\
    
    \noindent\emph{Step 2: simplifying the reflection terms.} It is straightforward to check that
    \begin{equation}
              R_k(u)=\frac i{{ 2}(2N)^{3/2}}(\mathbf{1}_{\Vee(k)} +\mathbf{1}_{\Wedge(k)}) \Big(\Phi_v\Big(\frac k{2N} \Big)-\Phi_w\Big(
             \frac k{2N}\Big)+O(N^{-3/2})\Big).
   	\end{equation}
   Since $\varphi_v,\varphi_w$ have compact supports in $(0,1)$, we can restrict ourselves to those $k\in [\delta 2N, (1-\delta) 2N]$ for some $\delta > 0$ small enough but fixed w.r.t.~$N$.
   By (1) of Lemma \ref{lemma:EstimatesmN} below, the quantity $O(N^{-3/2})$ has a negligible contribution in $\mathcal{L}_N^R \psi_\varphi(u)m_N(du)$. Furthermore, by (2) of Lemma \ref{lemma:EstimatesmN} it suffices to show that, as $N\to\infty$ and then $a\uparrow \infty$ the measure
   \begin{equation}\label{Eq:RN}\begin{split}
		(2N)^2\psi_\varphi(u)&\sum_{k\in [\delta 2N, (1 - \delta) 2N]} \sum_{j\in [-a\sqrt{2N},a\sqrt{2N}]} \frac i{{ 2}(2N)^{3/2}}(\mathbf{1}_{\Vee(k,j)} +\mathbf{1}_{\Wedge(k,j)})\\
		&\times\Big(\Phi_v\Big(\frac k{2N} \Big)-\Phi_w\Big(\frac k{2N}\Big)\Big) m_N(du)\;,
	\end{split}   
	\end{equation}
   converges weakly to
   \begin{equation}
   		\frac{i}{2} \psi_\varphi(u)\int_0^1 dr \frac{1}{\sqrt{2\pi r^3(1-r)^3}} \frac{(\varphi_v-\varphi_w)}{\sqrt{2}}(r) d\mBE(D \mid D(r) = 0)d\mBB(S)\;,
   \end{equation}
   where $S=(v+w)/\sqrt 2$ and $D=(v-w)/\sqrt 2$. Here, $\Vee(k,j)$ (resp. $\Wedge(k,j)$) denotes the event that $k$ is a contact point, the interfaces $\ttv,\ttw$ have a downward (resp. upward) corner there and $\ttv(k\pm1)=\ttw(k\pm1)=j$.
   \\
   
   \noindent\emph{Step 3: concatenation maps and decomposition of the invariant measure - the continuous case.} For any functions $u_1,u_2$ from $[0,1]$ into $\R^2$ and for any given $r\in (0,1)$, we define the concatenation map
   \[
   			T_{r} (u_1,u_2)(x) = \sqrt{r}\, u_1\Big(\frac{x}{r}\Big) \tun_{\{x\leq r\}} + \sqrt{1-r} \,u_2\Big(\frac{x-r}{1-r}\Big)\tun_{\{x>r\}}\;,\quad x\in [0,1]\;.
   \]
   Note that this is a function from $[0,1]$ into $\R^2$.\\
   We observe that under $m$, $S$ is a Brownian bridge on $[0,1]$ so that for any $r\in (0,1)$, $S(r) \sim\cN(0,r(1-r))$, and $D$ is an independent normalised Brownian excursion. Let $m_{(0,0)\to (1,y)}$ denote the probability measure on $E$ under which $S$ is a Brownian bridge between $(0,0)$ and $(1,y)$ and $D$ is an independent Brownian excursion. In particular $m_{(0,0)\to (1,0)} = m$. Let also $m_{(0,y)\to(1,0)}$ denote the probability measure on $E$ under which $S$ is a Brownian bridge between $(0,y)$ and $(1,0)$ and $D$ is an independent Brownian excursion.\\
   The law of $T_r(u_1,u_2)$ under the probability measure (on $E\times E$)
   $$ \int_{y\in\R} \frac1{\sqrt{2\pi r(1-r)}} e^{-\frac{y^2}{2r(1-r)}} m_{(0,0)\to (1,y)}(du_1) m_{(0,y)\to(1,0)}(du_2) \,dy\;,$$
   coincides with the law of $u$ under $\mBB(dS)\mBE(dD \mid D(r) =0)$.\\

   \noindent\emph{Step 4: concatenation maps and decomposition of the invariant measure - the discrete case.}
  	For any functions $u_1,u_2$ from $[0,1]$ into $\R^2$ and for any given integer $k\in\{2,\ldots,N-2\}$ we let $T_{N,k}^\Wedge (u_1,u_2)$ be the pair $(v,w)$ defined by
  	\begin{equation*}
  	(v,w)(x) :=\begin{cases}
  		\sqrt{\frac{k-1}{2N}}\, u_1\Big(\frac{x2N}{k-1}\Big)\quad &\mbox{ if } x \leq \frac{k-1}{2N}\;,\\
  		\sqrt{\frac{k-1}{2N}}\, u_1(1) + \Big(\frac{1-|x2N-k|}{\sqrt{2N}} , \frac{1-|x2N-k|}{\sqrt{2N}} \Big)\quad &\mbox{ if } \frac{k-1}{2N}< x\leq \frac{k+1}{2N}\;,\\
  		\sqrt{1-\frac{k+1}{2N}}\, u_2\Big(\frac{x-\frac{k+1}{2N}}{1-\frac{k+1}{2N}}\Big)\quad &\mbox{ if } \frac{k+1}{2N} < x \le 1\;.
  	\end{cases}
  \end{equation*}
  $T_{N,k}^\Vee (u_1,u_2)$ is defined similarly except that $\frac{1-|x2N-k|}{\sqrt{2N}}$ is replaced by $\frac{|x2N-k|-1}{\sqrt{2N}}$.\\
   Let $\nu_{(0,a)\to (n,b)}$ denote the uniform measure on the set of pairs $(\tv,\Tw)$ of lattice paths of length $n$ such that $\tv \ge \Tw$, $\tv(0)=\Tw(0)=a$ and $\tv(n)=\Tw(n)=b$. We consider the rescalings $v=\tv(\cdot n)/\sqrt{n}$ and $w=\Tw(\cdot n)/\sqrt{n}$. The law of $T_{N,k}^\Wedge(u_1,u_2)$ under the probability measure $\nu_{(0,0)\to (k-1,j)}(du_1) \nu_{(0,j)\to (2N-k-1,0)}(du_2)$ coincides with the law of $u$ under $m_N(du \mid \Wedge(k,j))$. The same assertion holds with $\Wedge$ replaced by $\Vee$.\\

   \noindent\emph{Step 5: end of proof.} Let $F:E\to\R$ be some bounded and continuous function and set $G:= F \psi_\varphi$. Note that $m_{N}(\Vee(k,j)))=m_{N}(\Wedge(k,j))$ (by symmetry, or see \eqref{eq:KmG} below). Using the identity in law from Step 4, we deduce that the integral of $F(u)$ against \eqref{Eq:RN} can be written
    \begin{align}
   	&\sqrt{2N}\sum_{k\in [\delta 2N, (1 - \delta) 2N]} \sum_{j\in [-a\sqrt{2N},a\sqrt{2N}]} \frac i{{ 2}} m_{N}(\Vee(k,j)) \Big(\Phi_v\Big(\frac k{2N} \Big)-\Phi_w\Big(\frac k{2N}\Big)\Big)\\
   	&\times  \int_E \sum_{X\in \{\Wedge,\Vee\}}  G(T_{N,k}^X(u_1,u_2)) \nu_{(0,0)\to (k-1,j)}(du_1) \nu_{(0,j)\to (2N-k-1,0)}(du_2)\;.
   \end{align}
	Note that only those $j$ whose parity matches with that of $k-1$ contribute in the sum. Using in turn Lemma \ref{lemma:EstimatesmN} (3) together with the change of variables $r=k/(2N)$ and $y=\sqrt{2} j / \sqrt{2N}$ this equals
	\begin{align}
		&\frac{i}{2} \int_{r=0}^1 \frac{dr}{\sqrt{2\pi r^3(1-r)^3}}  \int_{y=-\sqrt{2} a}^{\sqrt{2} a}\frac{e^{-\frac{y^2}{2r(1-r)}}}{\sqrt{2\pi r(1-r)}}dy \frac{\Big(\Phi_v\Big(\frac{\lfloor r2N\rfloor}{2N} \Big)-\Phi_w\Big(\frac{\lfloor r2N\rfloor}{2N}\Big)\Big)}{\sqrt 2}\\
		&\times J_N\Big(\lfloor r2N\rfloor,\lfloor \frac{\sqrt{2N} y}{\sqrt{2}}\rfloor,G \Big)(1+o(1))
	\end{align}
	where
	\begin{align}
		J_N(k,j,G) &:=\frac12\int_{E\times E}  \sum_{X\in \{\Wedge,\Vee\}} G(T_{N,k}^X(u_1,u_2)) \nu_{(0,0)\to (k-1,j)}(du_1) \nu_{(0,j)\to (2N-k-1,0)}(du_2)\;.
	\end{align}
	Since $|J_N(r,y,G)|$ is bounded uniformly over all $N$, $r$ and $y$, Lemma \ref{lemma:CVJ} below and the Dominated Convergence Theorem ensure that the whole expression we started from converges to	
		\begin{align*}
		&\frac{i}{2}\int_{E\times E} \int_{r=0}^1 \frac{dr}{\sqrt{2\pi r^3(1-r)^3}}  \int_{y=-\sqrt{2} a}^{\sqrt{2} a}\frac{e^{-\frac{y^2}{2r(1-r)}}}{\sqrt{2\pi r(1-r)}}dy \\
		&\times G(T_r(u_1,u_2)) \frac{(\varphi_v(r)-\varphi_w(r))}{\sqrt{2}} m_{(0,0)\to (1,y)}(du_1) m_{(0,y)\to (1,0)}(du_2)\;.
	\end{align*}
   Using the identity in law of Step 3, we deduce that this measure converges as $a\uparrow \infty$ towards
   		\begin{align*}
   	\frac{i}{2} \int_E\int_{r\in [\delta,1-\delta]} \frac{dr}{\sqrt{2\pi r^3(1-r)^3}} \frac{(\varphi_v(r)-\varphi_w(r))}{\sqrt{2}}  G(u) \mBB(S) \mBE(D \mid D(r)=0)\;,
   \end{align*}
	as required.
\end{proof}

\begin{lemma}\label{lemma:CVJ}
	For any bounded continuous function $G:E\to\R$, any $r\in (0,1)$ and $y\in\R$, $J_N(\lfloor r2N\rfloor, \lfloor\sqrt{2N} y/\sqrt{2}\rfloor,G)$ converges as $N\to\infty$ towards
	$$ \int_E G(T_r(u_1,u_2)) m_{(0,0)\to (1,y)}(du_1) m_{(0,y)\to (1,0)}(du_2)\;.$$
\end{lemma}
\begin{proof}
	     Let $\tilde{E}_N$ be the set of all pairs $(v,w)$ obtained by rescaling ordered lattice paths $\tv \ge \Tw$. Let us point out that, contrary to $E_N$, we do not pin these lattice paths at sites $0$ and $2N$ so that $E_N$ is a strict subset of $\tilde{E}_N$. Let $\pi_N:E\to \tilde{E}_N$ be a projection map as defined in Subsection \ref{sec:ageneral}.
		For $X\in \{\Wedge,\Vee\}$ and $r\in (0,1)$, it is straighforward to check that the maps $T_{N,\lfloor r2N \rfloor}^X$, $N\ge 1$ are uniformly (over $N$) Lipschitz from $\tilde{E}_N\times \tilde{E}_N$ into $\tilde{E}_N$. Of course, the Lipschitz constant depends on $r$ and blows up as $r$ approaches $0$ or $1$. Furthermore, for any given $y\in\R$, $T_{N,\lfloor r2N \rfloor}^X(\pi_N u_1,\pi_N u_2)$ converges to $T_r(u_1,u_2)$ $m_{(0,0)\to (1,y)}(du_1) m_{(0,y)\to (1,0)}(du_2)$-almost everywhere.\\
		A slight generalisation of Lemma \ref{lemma:CVInvMeas} shows the following weak convergences hold as $N\to\infty$
		\begin{align*}
			\nu_{(0,0)\to (\lfloor r2N\rfloor-1,\sqrt{2N} y/\sqrt{2})}(du_1) \to m_{(0,0)\to (1,y)}(du_1) \;,\\
			\nu_{(0,\sqrt{2N} y/\sqrt{2})\to (2N-\lfloor r2N\rfloor-1,0)}(du_2) \to m_{(0,y)\to (1,0)}(du_2)\;.
		\end{align*}
		We have all the ingredients to apply Lemma \ref{Lemma:EquiContCV} and deduce that $J_N(\lfloor r2N\rfloor, \lfloor\sqrt{2N} y/\sqrt{2}\rfloor,G)$ converges to
		$$ \int_E G(T_r(u_1,u_2)) m_{(0,0)\to (1,y)}(du_1) m_{(0,y)\to (1,0)}(du_2)\;.$$
\end{proof}

\begin{lemma}\label{lemma:EstimatesmN}
	Fix $\delta > 0$.
	\begin{enumerate}
		\item $$ \limsup_{N\to\infty} \sum_{k\in [\delta 2N, (1 - \delta) 2N]} \Big(m_N(\Vee(k)) + m_N(\Wedge(k))\Big) \sqrt{2N} < \infty\;.$$
		\item $$ \limsup_{a\to\infty}\limsup_{N\to\infty} \sum_{k\in [\delta 2N, (1 - \delta) 2N]} \sum_{j\notin [-a\sqrt{2N} , a\sqrt{2N} ]}\Big(m_N(\Vee(k,j)) + m_N(\Wedge(k,j))\Big)\sqrt{2N} = 0\;.$$
		\item Uniformly over all $k\in [\delta 2N, (1 - \delta) 2N]$ and all $j \in  [-(2N)^{2/3},(2N)^{2/3}]$ with the same parity as $k-1$, as $N\to\infty$
		\begin{align*}
				m_{N}(\Vee(k,j)) = m_{N}(\Wedge(k,j) = \frac1{2\pi (2N)^2} \Big(\frac{(2N)^2}{k (2N-k)}\Big)^{2} e^{-\frac{j^2}{2N} \frac{(2N)^2}{k(2N-k)}} (1+o(1))\;.
		\end{align*}
	\end{enumerate}
\end{lemma}
\begin{proof}
	By the Karlin-McGregor formula, see for instance~\cite[Theorem 2.5]{gorin2021lectures}, the number of configurations in $\Omega_N$ that have a contact point at $k$ with an upward (resp.~downward) corner at $k$ and are such that $\ttv(k-1)=j$ equals
	\begin{eqnarray}
		C(k-1,j)C(2N-k-1,j)\;,
	\end{eqnarray}
	where
	\begin{eqnarray}
		C(k,j)=\binom{k}{(k+j)/2}^2-\binom{k}{(k+j)/2-1}\binom{k}{(k+j)/2+1}
	\end{eqnarray}
	is the number of pairs or ordered simple random walk trajectories of length $k$, starting at $0$ and ending at $j$. As a consequence
	\begin{equation}
          \label{eq:KmG}
          m_{N}(\Vee(k,j)) = m_{N}(\Wedge(k,j)) = \frac{C(k-1,j)C(2N-k-1,j)}{C(2N,0)}\;.
        \end{equation}
	We fix $\delta > 0$ and we always work with $k\in [\delta 2N, (1 - \delta) 2N]$ so that $k\to\infty$ as $N\to\infty$. Define the function
	$$ f(x) = (1+x) \ln(1+x) + (1-x)\ln (1-x)\;,\quad x\in [-1,1]\;,$$
	and note that $f(x) = x^2 + O(x^4)$ as $x\to 0$. Note also that it is a convex function so that $f(x) \ge x^2$ over $[-1,1]$. A tedious computation shows that, provided $(k-1) - j \to \infty$ as $N\to\infty$, we have
	$$ C(k-1,j) \sim 2^{2k-1} \frac{4k(k-1)}{\pi ((k-1)^2-j^2)((k+1)^2-j^2)} e^{-(k-1) f(\frac{j}{k-1})}\;,\quad N\to\infty\;.$$
	A straightforward computation then leads to (3), noticing that in this statement $j$ is taken to be negligible compared to $N$ and therefore to $k$.\\
	To establish (1) and (2) we need to deal with values $j$ that can be arbitrarily close to $\pm (k-1)$ and $\pm (2N-k-1)$. Actually the contributions coming from $j\in \{-k+1,k-1,-2N+k+1,2N-k-1\}=:\cJ$ are easy to deal with since, in any of these four cases, either $C(k-1,j)=1$ or $C(2N-k-1,j)=1$. We therefore exclude these four values from now on. The same arguments as above show that there exists some constant $c>0$ such that for all $N$ large enough
	$$ C(k-1,j) \le c 2^{2k-1} \frac{4k(k-1)}{\pi ((k-1)^2-j^2)((k+1)^2-j^2)} e^{-(k-1) f(\frac{j}{k-1})}\;,$$
	and
	$$ C(2N-k-1,j) \le c 2^{4N-2k-1} \frac{4(2N-k)(2N-k-1)}{\pi ((2N-k-1)^2-j^2)((2N-k+1)^2-j^2)} e^{-(2N-k-1) f(\frac{j}{2N-k-1})}\;,$$
	and we deduce that (recall that $f(x) \ge x^2$)
	\begin{align*}
		&\sum_{k\in [\delta 2N, (1 - \delta) 2N]} \sum_{j\notin [-a\sqrt{2N} , a\sqrt{2N} ]\cup \cJ}\Big(m_N(\Vee(k,j)) + m_N(\Wedge(k,j))\Big)\sqrt{2N}\\
		&\lesssim 
		\sum_{k\in [\delta 2N, (1 - \delta) 2N]} \frac1{2N}\sum_{j\notin [-a\sqrt{2N} , a\sqrt{2N} ]}\frac1{\sqrt{2N}} e^{-\frac{j^2}{2N} \frac{1}{\delta(1-\delta)}}\;,
	\end{align*}
	which is bounded uniformly over all $N$ large and its limsup in $N$ goes to $0$ as $a\to\infty$. This establishes (2). The bound (1) can be deduced from (2) and (3).
\end{proof}

\subsection{Tightness}

We aim at checking Assumption \ref{Ass:Tight}. Recall the Sobolev space $H^{-\beta}([0,1],\R)$ introduced in Subsection \ref{Subsec:TightZRP}. We define $H^{-\beta}([0,1],\R^2)$ as the set of all pairs $(v,w)$ such that both $v$ and $w$ belong to $H^{-\beta}([0,1],\R)$. To show that the sequence $(u)_N$ is tight in $\bbD([0,\infty),H^{-\beta}([0,1],\R^2))$, it suffices to prove that the sequences $(v)_N$ and $(w)_N$ separately are tight in $\bbD([0,\infty),H^{-\beta}([0,1],\R))$ for any$\beta > 1$. The proof follows from the Lyons-Zheng decomposition~\cite{LyonsZheng}, exactly like in Subsection \ref{Subsec:TightZRP} for the Zero-Range Process so we do not present the details.\\

We have therefore completed the proof of the following result.

\begin{theorem}\label{Th:Pair}
	The discrete process $(u(t), t\ge 0)$, starting from its stationary measure $m_N$, converges in law to the process $(u(t),t\ge 0)$, starting from $m$, in the space \[\bbD([0,\infty),H^{-\beta}([0,1], \R^2))\] for any given $\beta > 1$.
\end{theorem}

\begin{remark} \label{rem:k>2} The dynamics of the two reflected
  interfaces can be easily generalised
  \cite{luby2001markov,wilson2004mixing} to $k\ge 2$ reflected
  interfaces, $\ttv_{(1)}\ge \dots\ge \ttv_{(k)}$ pinned at zero at
  their endpoints: $\ttv_{(i)}(0)=\ttv_{(i)}(2N)=0$ for all
  $1\le i\le k$. Assign the flip rates as follows. If interface
  $\ttv_{(i)}$ has an upward (resp. downward) corner at
  $1\le x\le 2N-1$, call $N_i(x)\ge0$ the number of interfaces with
  index $j$ strictly larger (resp. smaller) than $i$, such that
  $\ttv_{(j)}(y)=\ttv_{(i)}(y)$, $y=x,x+1,x-1$. In other words, the
  $j$ interfaces below (resp. above) $\ttv_{(i)}$ have a contact point
  with it at position $x$. Then, the corner of $\ttv_{(i)}$ is
  flipped, together with that of the other $j$ mentioned interfaces,
  with rate $(2N)^2/(2 (N_i(x)+1))$. See Fig. \ref{fig:k>2} for an
  illustration. The uniform measure over the collections of ordered
  interfaces $\ttv_{(i)},1\le i\le k$, is stationary and reversible
  for the process.  Also, as already mentioned above, the existence of
  a monotone global coupling under which the $L^1$ distance contracts
  is known \cite{wilson2004mixing}.  It is very likely that, modulo a
  certain amount of technical work (especially for the ``discrete
  characterization and convergence'', which requires applying the
  Karlin-McGregor formula for $k$ paths) one can use our method to
  show that, for $k$ fixed and $N\to\infty$, the stationary dynamics
  converges to a collection of $k$ reflected SPDEs, that generalises
  \eqref{eq:theSPDEs}.  A much more challenging situation is that
  where $k$ is of order $N$. In this case, the process can be seen as
  a dynamics for rhombus tilings of an $N\times N$ planar domain
  \cite{wilson2004mixing} (the interfaces $\ttv_{(i)}$ representing
  the level lines of the height function), and we expect that the
  scaling limit of the stationary process is a two-dimensional
  stochastic heat equation with additive noise. 
  \end{remark}

\begin{figure}[h]
  \centering
  \includegraphics[width=9cm]{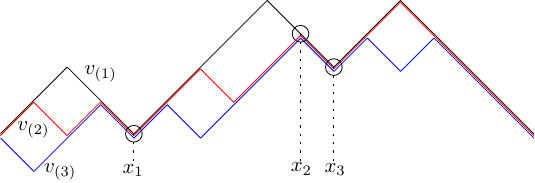}
  \caption{An example of allowed configuration for $k=3$ reflected interfaces. At positions $x_1,x_3$ there is a triple contact with a downward corner, and at $x_2$ a double contact with an upward corner. The rate at which $v_{(3)}$ flips upward at position $x_1$ (together with the other two interfaces) is $(2N)^2/6$, since $N_3(x_1)=2$. On the other hand, $v_{(2)}$ can also flip upward at $x_1$ (together with $v_{(1)}$) and it does so with rate $(2N)^2/4$.}
  \label{fig:k>2}
\end{figure}
  
\section{$\nabla\phi$ interfaces with convex potential in one dimension}

\label{sec:nablaphi}

To conclude, let us present yet another example of interesting Markov process for which our general method should apply. Given a convex function $V:\mathbb R\mapsto \mathbb R$ bounded from below, consider the
one-dimensional $\nabla\phi$  interface model described by the height function $h=(h_i)_{i=0,1,\dots,N}$ with
boundary conditions $h_0=h_N=0$ and with Boltzmann-Gibbs measure given as
\begin{eqnarray}
  m_N(d h)\propto e^{-\sum_{i=1}^N V(h_i-h_{i-1})}\prod_{i=1}^{N-1}d h_i
\end{eqnarray}
(the potential $V$ should increase at least linearly at $\pm\infty$, so that the measure is normalisable).
One can consider two other variants of this model:
\begin{itemize}
\item the one with ``hard wall'', where $m_N$ is restricted to the subspace where $h_i\ge0$ for every $i$;
  
\item the discrete version (with or without the ``hard wall'' constraint $h_i\ge0$) where the heights $h_i$ are restricted to integer values and the Lebesgue measure $\prod_{i=1}^{N-1}d h_i$ is replaced by the counting measure.
\end{itemize}
The following dynamics (Gibbs sampler) is reversible with respect to $m_N$: with rate $N^2$, the height $h_i$, $1\le i\le N-1$ is updated to a new value sampled from $ m_N(\cdot|h_{i\pm1})$, i.e., the equilibrium measure conditioned to the value of the heights in neighboring positions.

Let us point out that the dynamics with continuous heights and hard wall constraint can be seen as a Poissonian version of the dynamics considered by Funaki and Olla~\cite{FunaOlla}.

For the dynamics with continuous heights and no hard-wall constraint,
\cite{caputo2022spectral} identified precisely the spectral gap and
proved the occurrence of total variation cut-off, under the assumption
that $V$ is of at most polynomial growth at infinity. For the case
with discrete heights and hard-wall constraint, the mixing time was
identified in \cite[Sec. 3]{martinelli2009mixing}, up to a global
prefactor.

We believe that the method of the present paper allows to prove that,
under the same assumptions on $V$, the stationary process of suitably
rescaled height fluctuations
$\tilde h_N(t,x)=N^{-1/2}h_{\lfloor x N\rfloor}(t), x\in[0,1],t\ge0$
converges in distribution to either the stochastic heat equation with
additive noise (if the hard wall constraint is absent) or the
Nualart-Pardoux equation \eqref{eq:NP} (if the hard wall constraint is
present). We do not expect the actual proof to be harder than or
significantly different from that of either the zero-range process or
of the pair of reflected interfaces, given above, but we preferred not
to work it out, in order to keep this work within a reasonable length.

Let us just mention that it is easy to check that, because of the
assumed convexity of $V$, these interface dynamics admit a global
monotone coupling, and that under this coupling, the $L^1$ distance
between configurations contracts (this holds both with and without
hard-wall constraint).  This allows to obtain equicontinuity of the
sequence of resolvants. The technical work that is left concerns the
discrete integration by parts, that involves arguments based on the
local limit theorem for random walk (and for random walks conditioned
to be positive) \cite{bryn2006functional,caravenna2008invariance}.

\subsection*{Data availability} We do not analyse or generate any datasets.

\subsection*{Competing interests} The authors have no competing interests to declare that are relevant to the content of this article.

\bibliographystyle{Martin}
\bibliography{library}

\end{document}